\documentclass[twoside]{amsart}
\usepackage{color,graphicx,wrapfig}
\usepackage{eucal}
\usepackage{amssymb}
\usepackage{amsmath}
\usepackage{amssymb}
\newtheorem{theorem}{Theorem}[section]
\newtheorem{lemma}[theorem]{Lemma}
\newtheorem{corollary}[theorem]{Corollary}
\newtheorem{proposition}[theorem]{Proposition}
\theoremstyle{definition}

\newtheorem{example}[theorem]{Example}
\newtheorem{remark}[theorem]{Remark}

\newtheorem{definition}[theorem]{Definition}

\allowdisplaybreaks

\newcommand{\art}[6]{{\sc #1, \rm #2, \it #3 \bf #4 \rm (#5), \mbox{#6}.}}
\newcommand{\artin}[3]{{\sc #1, \rm #2,  in #3.}}

\newcommand{\book}[3]{{\sc #1, \it #2, \rm #3.}}
\newcommand{\AND}{{\rm and }}

\def\R{{\mathbb R}}

\def\N{\mathbb{N}}
\def\Q{\mathbb{Q}}
\def\-{\setminus}

\def\P{\mathcal{P}}

\def\S{\mathcal{S}}

\def\L{\mathcal{L}}

\def\M{\mathcal{M}}
\def\MM{\mathbb{M}}
\def\X{\mathcal{X}}

\def\RR{\mathcal{R}}
\def\WR{\widetilde{\mathcal{R}}}
\def\K{\mathcal{K}}

\def\S{\mathcal{S}}
\def\min{{\rm min}}
\def\supp{{\rm supp}}
\def\max{{\rm max}}
\def\vae{\varepsilon}

\def\vae{\varepsilon}

\begin{document}
\setcounter{page}{1}
\title[Sobolev Spaces]{A new approach to Sobolev spaces in metric measure spaces}
\author[T. Sj\"{o}din]{Tomas Sj\"{o}din}
%\date{\today}
\address{
Department of Mathematics\newline
\indent Link\"o{}ping University\newline
\indent 581 83, Link\"o{}ping, Sweden.}
\email{tomas.sjodin@liu.se}
%\subjclass{}

\begin{abstract}
Let $(X,d_X,\mu)$ be a metric measure space where $X$ is locally compact and separable and $\mu$ is a Borel regular measure such that $0 <\mu(B(x,r)) <\infty$ for every ball $B(x,r)$ with center $x \in X$ and radius $r>0$. We define $\X$ to be the set of all positive, finite non-zero regular Borel measures with compact support in $X$ which are dominated by $\mu$, and $\M=\X \cup\{0\}$. By introducing a kind of mass transport metric $d_{\M}$ on this set we provide a new approach to first order Sobolev spaces on metric measure spaces, first by introducing such for functions $F : \X \rightarrow \R$, and then for functions $f:X \rightarrow [-\infty,\infty]$ by identifying them with the unique element $F_f : \X \rightarrow \R$ defined by the mean-value integral:
$$F_f(\eta)= \frac{1}{\|\eta\|}\int f \,d\eta.$$
In the final section we prove that the approach gives us the classical Sobolev spaces when we are working in open subsets of Euclidean space $\R^n$ with Lebesgue measure.
\end{abstract}

\maketitle

\renewcommand{\thefootnote}{\fnsymbol{footnote}} 
\renewcommand{\thefootnote}{\arabic{footnote}} 
\section{Introduction}
Suppose $(X,d_X,\mu)$ is a metric measure space and $1 \leq p <\infty$. If we want to introduce a first order Sobolev-type space, analogous to the classical Sobolev spaces $H^{1,p}(X)$ when $X$ is an open subset of $\R^n$, $d_X$ the Euclidean distance and $\mu$ the Lebesgue measure, then there is by now a few approaches available, most notably that based on upper gradients, which were introduced by  Heinonen and Koskela~\cite{HeKo98}, such as first studied by Shanmugalingam in \cite{Sh-rev}. By now there are (at-least) two good books which treat this approach in detail, first \cite{AJB} by  Bj\"o{}rn and Bj\"o{}rn and very recently \cite{HKSTbook} by Heinonen,  Koskela,  Shanmugalingam and Tyson.

Apart from the Newtonian spaces there are alternative definitions of Sobolev spaces on metric measure spaces worth mentioning. Early approaches are due to  Haj\l asz in \cite{HAJ} and Korevaar-Schoen (a version directly comparable to this article of the latter approach seems first to have been developed in \cite{Ko-Ma}). Other approaches can be found in \cite{CHE} by Cheeger and \cite{SHV} by Shvartsman. There has also been some axiomatic treatments (see e.g. \cite{GT1,T1}). The survey articles \cite{Haj03, HEI} are also worth mentioning as well as the book \cite{HeKiMa} which treats weighted Sobolev spaces on $\R^n$.

The idea of upper gradients is based on the well-known formula
\begin{equation}\label{newt1}
|u(\gamma(s))-u(\gamma(0))| \leq \int_0^s g(\gamma(t))\, dt
\end{equation}
which holds for every smooth function in $\R^n$ and every rectifiable curve $\gamma$ parametrized by arc-length, in case we put $g = |\nabla u|$. In a metric space we do not have a direct substitute for $\nabla u$, but one then says that a Borel measurable function $g$ is an upper gradient of $u$ if the above formula holds for all curves. In case $g \in L^p(X)$ one says that $g$ is a $p$-upper gradient of $u$. If $u \in L^p(X)$ has an upper gradient which also belongs to $L^{p}(X)$, then one says that $u$ belongs to the {\bf Newtonian space} $N^{1,p}(X)$, and give it the norm
$$\|u\|_{N^{1,p}(X)} = \left( \int |u|^p \, d\mu + \inf_g \int g^p \, d\mu\right)^{1/p},$$
where the infimum is taken over all upper gradients $g$ of $u$.

For many questions it is desirable to have a minimal upper gradient $\tilde{g}_u$ of $u$ such that the above infimum is attained. As it turns out however such a minimal upper gradient does not always exist, and we are forced to introduce the somewhat technical concept of curve modulus to introduce weak upper gradients which satisfies inequality (\ref{newt1}) for ``almost every'' curve, which is given a precise meaning thorough the concept of curve modulus. It turns out that there is a unique, as an element in $L^p$, $p$-weak upper gradient $\tilde{g}_u$ of $u$, if $u$ has an upper gradient in $L^p$.

The aim of this paper is to look at an alternative approach. We do not know in general how these spaces are related to the Newtonian ones, but at the very least we do indeed get the classical Sobolev spaces  in case $X$ is an open subset of $\R^n$ with Lebesgue measure (which in turn are equivalent to the Newtonian spaces in this setting). 

To outline the approach assume that $(X,d_X,\mu)$ is a metric measure space, where $X$ is separable and locally compact, and $0< \mu(B)<\infty$ for every ball $B \subset X$. 
Let $\M$ denote the set of all (non-negative Radon) measures on $X$ which are dominated by $\mu$ and have compact support, and let $\X = \M \setminus \{0\}$.
In Section \ref{METRIC} we introduce a metric $d_{\M}$ on the set $\M$, and we give $\X$ the induced metric.
The idea is to first look at real-valued functions $F$ on $\X$, and to relate functions $f$ on $X$ to such by the mean-value integral as follows. If $\eta \in \X$ and $f$ is a locally integrable function on $X$, then we define
$$F_f(\eta) = \frac{1}{\|\eta\|} \int f \, d \eta,$$
where $\|\eta\|= \int \, d\eta$ is the total mass of $\eta$.  
It is worthwhile to remark that if $f$ is a locally integrable function on $X$, then point values are not really well defined (in the sense that we may have several representatives which are equal almost everywhere), but the value of $F_f$ on elements in $\X$ is always well defined and finite. So the elements of $\X$ have a similar role to test functions. This is perhaps the major motivation for this type of approach. In some sense $L^p$-functions are more natural to think of as certain type of functions on $\X$ rather than $X$, and hence it seems natural to see to what extent one can carry the calculus to this set in a natural way.

In Section \ref{LP} we introduce a norm $\|\cdot\|_{\L^p(\X)}$ on the set of extended real-valued functions
on $\X$, and we let $\L^p(\X)$ denote the set of such functions for which this expression is finite. In case $f \in L^p(X)$, then $\|F_f\|_{\L^p(\X)} = \|f\|_{L^p(X)}$.
It is also worthwhile to remark that the definition of the norm $\|\cdot\|_{\L^p(\X)}$ does not depend on the metric $d_{\M}$.

In Section \ref{upgrad} we introduce upper gradients $r_F: \X \rightarrow [0,\infty]$ for real-valued functions  $F:\X \rightarrow \R$. This definition is a pointwise (in $\X$) local definition, and this definition does not depend on an integrability exponent, unlike the definition of minimal $p$-weak upper gradients (it seems however to be an open question to what extent $\tilde{g}_u$ actually depends on the exponent $p$ in general). As it turns out, in case $F$ has a representative of the form $F_f$ for some function $f \in L^1_{\rm loc}(X)$ then also $r_F$ has a representative of the form $F_{g_f}$ for a function $g_f \in L^1_{\rm loc}(X)$.
In particular these upper gradients $g_f$ satisfies the strong locality property (see Theorem \ref{lattice1}). This is also true for the minimal weak upper gradients in the Newtonian theory, but the corresponding result does not hold in the approaches by Haj\l asz or Korevaar-Schoen for instance.

Then we introduce the Sobolev-type spaces and norms $\S^{1,p}(\X)$ and $\|\cdot\|_{\S^{1,p}(\X)}$ respectively in Section \ref{SOBOLEV}. Then, in Section \ref{Sobolev2}, we also introduce the space $S^{1,p}(X)$ as those functions $f \in L^p(X)$  such that $F_f$ belongs to $\S^{1,p}(\X)$. These will be our analogues of Sobolev spaces on $X$.

In Section \ref{RN} we prove that in case $X$ is an open subset of $\R^n$, $d_{X}$ is the usual Euclidean metric and $\mu$ is the Lebesgue measure, then the classical Sobolev space $H^{1,p}(X)$ and the space $S^{1,p}(X)$ coincides, and the norms are the same. Indeed we have $g_f=|\nabla f|$ for such functions.

We end the article with some final remarks about the particular choices made in the article and also mention questions for future research.

It is also worthwhile already here to point out that the development of the theory over $\X$ depends only on some basic properties of rectifiable curves in that space, and not directly of the underlying space $X$, and even the spaces $S^{1,p}(X)$ has an analogue $S^{1,p}(\X)$ defined on $\X$ in a way that need not make reference to $X$ either. Although the above is not emphasised in this article, these facts opens up the possibility to develop a theory which is point-free such as in pointless topology for instance. 

\section*{Acknowledgements} The author wishes to thank Professors Anders and Jana Bj\"o{}rn for valuable discussions and suggestions.

\section{List of notation}
\noindent{\bf Some special sets}
\begin{itemize}
\item $\R$: the set of real (finite) numbers,
\item $\Q$: the set of rational numbers,
\item $\N$: the set of natural numbers $\{1,2,3,\ldots\}$,
\item $\mathbb{Z}$: the set of integers,
\end{itemize} 
\noindent{\bf Some lattice notation}
\begin{itemize}
\item If $a,b \in [-\infty,\infty]$, then $a \wedge b = \min\{a,b\}$ and $a \vee b = \max\{a,b\}$,
\item If $f,g$ are extended real-valued functions, then $f \wedge g$ and $f \vee g$ denotes their pointwise minimum and maximum respectively.
\end{itemize}
\noindent{\bf Some notation related to general metric spaces}\smallskip\newline
Below we let $(Y,d_Y)$ be a metric space (i.e. $Y$ is a set and $d_Y$ is a metric on $Y$).
\begin{itemize}
\item $B(y,r)$: ball with center $y$ and radius $r$, 
\item $\widetilde{\RR}(Y)$ :  rectifiable paths $\gamma : [0,b_{\gamma}] \rightarrow Y$  subparametrized by arc-length (i.e. $\gamma$ is $1$-Lipschitz),
\item $\RR(Y)$ :  rectifiable paths parametrized by arc-length,
\item $l_{\gamma}$  : length of a curve $\gamma$,
\item  $\breve{f}$ : upper semicontinuous regularization of the function $f$ along curves,
\item If $A \subset Y$ and $\vae \geq 0$ then $A_{\vae} = \{y \in Y : {\rm dist}(y,A) \leq \vae\}$. 
\end{itemize}
\noindent {\bf Some notation related to metric measure spaces}\smallskip\newline
Below $(X,d_X,\mu)$ will always denote a metric measure space. More precisely, $(X,d_X)$ is a metric space, and $\mu$ is assumed to be a Borel regular measure such that $0< \mu(B(x,r)) < \infty$ for every ball $B(x,r) \subset X$. The space $X$ is furthermore assumed to be locally compact and separable.
\begin{itemize}
\item $L^p(X)$: $p$-th power integrable functions on $X$ with respect to $\mu$, where $p \in [1,\infty]$,
\item $L^p_{\rm loc}(X)$: local $L^p$-spaces on $X$. 
\end{itemize}
\noindent{\bf  Notation related to $(\M,d_{\M})$}
\begin{itemize}
\item $\P$ : all non-negative finite Borel measures with compact support in $X$,
\item If $\eta \in \P$ then $\|\eta\| = \int\, d\eta$ denotes the total mass of $\eta$,
\item $\M =\{\nu \in \P: 0 \leq \nu \leq \mu\}$,
\item $\X = \M \setminus \{0\}$,
\item If $f \in L^1_{\rm loc}(X)$ then $F_f : \X \rightarrow \R$ is defined by 
$$F_f(\eta) = \frac{1}{\|\eta\|} \int f \,d\eta,$$
\item If $\eta,\nu \in \M$ then 
$$\eta \wedge \nu = \min\left\{\frac{d \eta}{d\mu}, \frac{d \nu}{d\mu}\right\} \mu, \quad \eta \vee \nu = \max\left\{\frac{d \eta}{d\mu}, \frac{d \nu}{d\mu}\right\} \mu.$$
\end{itemize}
Let $h: [0,\infty) \rightarrow [0,\infty)$ be a strictly increasing continuous function such that $h(0)=\lim_{\vae \rightarrow 0^+} \frac{h(\vae)}{\vae} =0,$ 
$h(\vae_1) + h(\vae_2) \leq h(\vae_1 + \vae_2) \textrm{ for all } \vae_1,\vae_2 \in [0,\infty)$
and $\lim_{\vae \rightarrow \infty} h(\vae)=\infty$.

If $\nu,\eta \in \M$ and $\vae,\delta >0$ then
\begin{align*}
& \Gamma_{\vae,\delta}(\nu,\eta) = \left\{ (\nu_i,\eta_i)_{i \in \MM}: \begin{array}{l} 
\MM \textrm{ is at most countable},\\
\nu_i,\eta_i \in \M \textrm{ for each } i \in \MM\\
\nu = \sum_{i \in \MM} \nu_i,\quad \eta = \sum_{i \in \MM}\eta_i,\\
\sum_{i\in \MM}  \bigl|\|\nu_i\|-\|\eta_i\|\bigr| \leq \delta,\\
\textrm{diam}(\supp(\nu_i) \cup \supp(\eta_i)) \leq \vae. \end{array}\right\},\\
& \Gamma_{\vae}(\nu,\eta)=\Gamma_{\vae,h(\vae)}(\nu,\eta),\\
&d_{\M}(\nu,\eta) := \inf \{\vae: \Gamma_{\vae}(\nu,\eta) \ne \emptyset\}.
\end{align*}

Rectifiable curves in $\WR(\M)$ are maps $\eta : [0,b_{\eta}] \rightarrow \M$, so for every $s \in [0,b_{\eta}]$ $\eta(s)$ is a measure in $\M$, and it turns out that every such curve satisfies that ${\bigcup_{s \in [0,b_{\eta}]} {\rm supp} (\eta(s))}$ is compact, and that $\|\eta(s)\|$ is constant.\smallskip\newline
\noindent{\bf  Notation related to $\L^p(\X)$}\smallskip\newline
For a fixed $p \in [1,\infty)$ and a function $F : \X \rightarrow [-\infty,\infty]$ we introduce the norm
\begin{equation*}
\|F\|_{\L^p(\X)} = \sup\left\{\left(\sum_{i=1}^k |F(\eta_i)|^p \|\eta_i\|\right)^{1/p}:  \eta_i \in \X, \,  \supp(\eta_i) \cap \supp(\eta_j) = \emptyset \textrm{ if } i \ne j\right\}.
\end{equation*}
\begin{itemize}
\item $\L^p(\X) = \{F : \X \rightarrow \R: \|F\|_{\L^p(\X)} < \infty\}$, 
\item  $L^p(\X) = \{F_f : f \in L^p(X)\}$,
\item $\L^p_{\rm loc}(\X)$, $L^p_{\rm loc}(\X)$: local versions of $\L^p(\X)$ and $L^p(\X)$ (see section \ref{lploc}).
\end{itemize} 
\noindent{\bf  Notation related to upper gradients}\smallskip\newline
For a function $F: \X \rightarrow \R$ and a number $\vae >0$ we introduce:
$$r_F^{\vae}(\eta) := \sup\left\{\frac{|F(\nu(s)) -F(\eta)|}{s} : 
 \nu \in \WR(\X) \textrm{ such that } \nu(0)=\eta \textrm{ and } 0< s < \vae \wedge b_{\nu} \right\},$$
and then we define
$$r_F(\eta)= r_F^{0}(\eta) = \lim_{\vae \rightarrow 0} r_F^{\vae}(\eta).$$
For any element $\eta \in \WR(\X)$ we have
$$|F(\eta(s)) -F(\eta(0))| \leq \int_0^s \breve{r}_F(\eta(t))\,dt.$$

In case $f \in L^1_{\rm loc}(X)$, and $r_{F_f} \in \L^1_{\rm loc}(\X)$, then there is an a.e. unique function $g_f \in L^1_{\rm loc}(X)$ such that
$$r_{F_f} =F_{g_f}.$$
 
\noindent{\bf  Notation related to $\S^{1,p}(\X)$}\newline
\begin{itemize}
\item $\|F\|_{\S^{1,p}(\X)} : = (\|F\|_{\L^p(\X)}^p + \|r_F\|_{\L^p(\X)}^{p})^{1/p},$
\item $\S^{1,p}(\X) = \{ F \in \L^p(\X) : \|F\|_{\S^{1,p}(\X)} < \infty\},$
\item $S^{1,p}(\X) = \{ F \in L^p(\X) : \|F\|_{\S^{1,p}(\X)} < \infty\},$
\item $S^{1,p}(X) =\{f \in L^p(X): F_f \in S^{1,p}(\X)\}.$
\end{itemize}
\section{Preliminaries}
Given a metric space $(Y,d_Y)$ we denote by $B(y,r)$ the ball with center $y$ and radius $r$ (where the space $Y$ should be understood from the context). 
For a set $A \subset Y$ and $\vae \geq 0$ we also introduce
$$A_{\vae} = \{y \in Y: {\rm dist}(y,A) \leq \vae\}.$$
It is clear that $A_{\vae}$ is closed, $A_0 = \overline{A}$ and $(A_{\vae})_{\delta} \subset A_{\vae+\delta}$. Furthermore we note that in case $\vae_n$ decreases to $\vae$ as $n \rightarrow \infty$, then $A_{\vae_n}$ decreases to $A_{\vae}$ as $n \rightarrow \infty$. We also have for $A \subset C$ and $\vae \leq \delta$ that 
$A \subset A_{\vae} \subset C_{\vae} \subset C_{\delta}.$  Finally, if $K \subset Y$ is compact and $Y$ is locally compact, then there is $\vae >0$ such that $K_{\vae}$ is also compact.

A {\it rectifiable curve} $\gamma$ is a map $\gamma : [a,b] \rightarrow Y$ where $-\infty <a \leq b < \infty$   such that the length $l_{\gamma} < \infty$, where the length is defined by
$$l_{\gamma} = \sup\left\{\sum_{i=0}^k d_Y (\gamma(a_{i+1}),\gamma(a_i)): a \leq a_0 \leq a_1 \leq a_2 \leq \ldots \leq a_{k+1} \leq b\right\}.$$
We say that a rectifiable curve $\gamma$ is subparametrized by arc-length if the map $\gamma$ is $1$-Lipschitz, i.e. if for every $a\leq s\leq s+t \leq b$ we have
$$d_Y(\gamma(s+t),\gamma(s)) \leq t.$$
In particular, if $\gamma$ is subparametrized by arc-length then for every $a\leq s\leq s+t \leq b$ have
$$t \geq \sup\left\{ \sum_{i=0}^k d_Y(\gamma(r(a_{i+1})),\gamma(r(a_i))): s=a_0\leq a_1 \leq \ldots \leq a_{k+1}=s+t\right\}$$
(that is, the length of the curve $\gamma|_{[s,s+t]}$ is at most $t$).
We let $\WR(Y)$ denote the set of rectifiable curves subparametrized by arc-length. Unless otherwise stated we assume in this case that $a=0$ and $b=b_{\gamma} \geq 0$ in the sequel.
 
A rectifiable curve may always be parametrized by arc-length in the sense that there is an increasing function $r: [0,l_{\gamma}] \rightarrow [a,b]$ such that for every pair of non-negative real numbers $s,t$ such that $0 \leq s \leq s+t \leq l_{\gamma}$ we have
$$t=\sup\left\{ \sum_{i=0}^k d_Y(\gamma(r(a_{i+1})),\gamma(r(a_i))): s=a_0\leq a_1 \leq \ldots \leq a_{k+1}=s+t\right\}.$$
In case $a=0$, $b=l_{\gamma}$ and $r(x)=x$ for each $x$, then we say that $\gamma$ is parametrized by arc-length. We denote the set of all rectifiable curves parametrized by arc-length by $\RR(Y)$. For reasons that will be made clear later we will mainly work with the set $\WR(Y)$ rather than $\RR(Y)$ in this article.
\begin{remark}
If $\gamma \in \RR(Y)$, then we write $\int_0^{l_{\gamma}} f(\gamma(s))\, ds$ for the path integral of the function $f$ over $\gamma$. 

We should also warn the reader already here that the points in our spaces will typically later be measures, so the notation $\int f \,d \gamma$ will not be used for path integrals since this can be misunderstood.

Indeed in this context $\int f \,d \gamma(s)$ would denote the integral of $f$ with respect to the measure $\gamma(s)$.
\end{remark}
{\bf From now on by a curve we will always mean a rectifiable curve unless otherwise stated.}
\begin{lemma} \label{fundth}
Suppose $f: [0,a] \rightarrow \R$ and $g : [0,a] \rightarrow [0,\infty)$ are such that $g$ is upper semicontinuous and 
$$\limsup_{s \rightarrow t} \left| \frac{f(s)-f(t)}{s-t}\right| \leq g(t) \textrm{ for all  } t \in [0,a],$$
then
$$|f(s)-f(0)| \leq \int_0^s g(t)\, dt \textrm{ for all  } s \in [0,a].$$
\end{lemma}

The above no doubt well-known fact will be extensively used, and in particular it will be important when we study function restrictions to rectifiable curves. Some more terminology associated with (rectifiable) curves are as follows.
\begin{definition}
We say that a function $f :Y \rightarrow \R$ is 
\begin{itemize}
\item[(a)] {\it continuous along curves} if $\lim_{s \rightarrow t} f(\eta(s))=f(\eta(t))$ for all $\eta \in \WR(Y)$,\\
\item[(b)] {\it upper semicontinuous along curves} if $\limsup_{s \rightarrow t} f(\eta(s)) \leq f(\eta(t))$ for all $\eta \in \WR(Y)$,\\
\item[(c)] {\it lower semicontinuous along curves} if $\liminf_{s \rightarrow t} f(\eta(s)) \geq f(\eta(t))$ for all $\eta \in \WR(Y)$.
\end{itemize}
\end{definition}
Given a function $f:Y \rightarrow \R$ we introduce the {\it upper semicontinuous regularization over curves} $\breve{f}$ of $f$ as
$$\breve{f}(y)=\lim_{s \rightarrow 0} \left(\sup \{f(\gamma(r)): 0 \leq r <s, \gamma \in \WR(Y), \gamma(0)=y\}\right).$$
Note that for any $\delta >0$ and $\vae>0$ there are a curve $\gamma$ and $r \in [0,\delta)$ such that
$$|\breve{f}(y) - f(\gamma(r))| <\vae.$$
Indeed we even have the following result
\begin{lemma}\label{maxalc1}
If $f : Y \rightarrow \R$ and $y \in Y$ then there is a curve $\nu \in \WR(Y)$ such that 
$$\breve{f}(y) = \limsup_{s \rightarrow 0} f(\nu(s)).$$
\end{lemma}
\begin{proof}
In case $\breve{f}(y)=f(y)$, then we may simply let $\nu$ be the constant curve with value $y$. Otherwise we may by definition inductively choose sequences $\delta_n$ and $\vae_n$ decreasing to zero and curves $\nu_n \in \WR(Y)$ such that
\begin{itemize}
\item[(1)] $0<\delta_n \leq \vae_n$,
\item[(2)] $\nu_n(0)= y$,
\item[(3)] $\breve{f}(y) \leq f(\nu_n(\delta_n))+\vae_n$,
\item[(4)] $\vae_{n+1} \leq 2^{-n} \delta_n$.
\end{itemize}
Let
$$k_n= \delta_n + 2\sum_{j=1}^{n-1} \delta_j, \quad b=2\sum_{j=1}^{\infty} \delta_j.$$
We define $\gamma: [0,b] \rightarrow Y$ such that 
$$\gamma(t)= \left\{ 
\begin{array}{ll} 
\nu_n(t-k_n+\delta_n), & t \in [k_n-\delta_n,k_n)\\
\nu_n(k_n+\delta_n - t), & t \in [k_n,k_n+\delta_n)\\
y, & t=b.\end{array}\right.$$
Finally put $\nu(t)=\gamma(b-t)$. It is easy to verify that $\nu \in \WR(Y)$ ($\gamma$ simply consists of rectifiable curves subparametrized by arc-length going back to forth from $y$ and then patched together). If we define 
$$r_n=\delta_n + 2\sum_{j=n+1}^{\infty} \delta_j,$$
then 
$$\nu(r_n)=\gamma(b-r_n) =\gamma\left(\delta_n+2\sum_{j=1}^{n-1}\delta_j\right) =\gamma(k_n)=\nu_n(\delta_n).$$
Hence we get
\begin{align*}
&\breve{f}(y) \leq \limsup_{n \rightarrow \infty} (f(\nu(r_n)) + \vae_n) \leq \limsup_{s \rightarrow 0} f(\nu(s)) + \limsup_{n \rightarrow \infty} \vae_n \leq \breve{f}(y).
\end{align*}
\end{proof}
Throughout the article we let $(X,d_X,\mu)$ be a fixed metric measure space such that 
\begin{align*}
&0<\mu(B(x,r))< \infty  \textrm{ for all balls } B(x,r) \textrm{ of radius } r \textrm{ and center }x, \\
& X \textrm{ is locally compact,}\\
& X \textrm{ is  separable.}
\end{align*}
(Actually the last part is a consequence of the first two assumptions since $X={\rm supp}(\mu)$.)

For $p \in [1,\infty]$ we use the notation
$$L^p(X)$$
to denote the class of all Borel measurable extended real-valued functions $f$ such that
\begin{align*}
\|f\|_{L^p(X)} &= \left(\int |f|^p \,d \mu\right)^{1/p} < \infty \quad (p \in [1,\infty)),\\
 \|f\|_{L^{\infty}(X)} &= {\rm ess\, sup}_{x \in X} |f(x)| < \infty. 
\end{align*} 
The spaces $L^p_{\rm loc}(X)$ are also defined as usual. 

If $Y$ is a set and $f,g: Y \rightarrow \overline{\R}$ then we introduce the following lattice notation:
$$f \vee g(x) = \max\{f(x),g(x)\},\quad f \wedge g(x) = \min\{f(x),g(x)\},$$
which makes the set of all such functions into a lattice (if we restrict attention to real-valued functions these forms a vector lattice).

A measure $\eta$ on $X$ will always refer to a non-negative Borel measure such that $\eta(B(x,r)) < \infty$ for all balls $B(x,r)$.  For any measure $\eta$ we also let $\|\eta\|$ denote its total mass.
These measures on $X$ also carries a natural partial order, and it is well-known to be a lattice. For any Borel measures $\eta,\nu$ on $X$ we denote their least upper bound and greatest lower bound by
$\eta \vee \nu$ and  $\eta \wedge \nu$ respectively.

\begin{lemma}\label{K_t}
Suppose $K(t) \subset X$ is compact for each $t \in [0,T]$. Suppose furthermore that for any $s,t \in [0,T]$ we have with $\vae =|s-t|$
$$K(t) \subset K(s)_{\vae}.$$
Then 
$$ \bigcup_{t \in  [0,T]} K(t)$$
is compact.
\end{lemma}
\begin{proof}
Let 
$$f(s) =  \sup \left\{\vae : K(s)_{\vae} \textrm{ is compact}\right\}.$$
Clearly $f(s) >0$ for each $s$. 
Let $c= \frac{1}{2}\inf \{f(s) : s \in [0,T]\}$. Suppose $c=0$, then there is a sequence $s_n$ converging to some $s$ in $[0,T]$ such that $f(s_n) \rightarrow 0$ as $n \rightarrow \infty$.
But if we put $\delta =f(s)/4>0$, then for $|s_n - s|< \delta$ we get
$$K({s_n})_{\delta} \subset (K({s})_{\delta})_{\delta} \subset K(s)_{2 \delta},$$
which by definition is compact. Hence we get a contradiction, and we see that indeed $c>0$. But now we get
$$\bigcup_{t \in [0,T]} K(t) \subset K(0)_c \cup K(c)_c \cup \ldots \cup K(nc)_c$$
where $nc \leq T <(n+1)c$. The right hand side is compact, so it only remains to show that $\bigcup_{t \in [0,T]} K(t)$ is closed.
So suppose that $x \in \overline{\bigcup_{t \in [0,T]} K(t)}$. Then by definition there is for each $n$ a point $x_n \in \bigcup_{t \in [0,T]} K(t)$ such that $d_X(x_n,x) \leq 1/n$, and then there are $t_n \in [0,T]$ such that $x_n \in K(t_n)$. We may assume, by passing to a subsequence, that $t_n$ converges to $t$ as $n \rightarrow \infty$. If $\vae >0$, then for every $n$ so large that $|t-t_n|+1/n <\vae$ we have
$$x \in K(t_n)_{1/n} \subset K(t)_{\vae}.$$
Since $\vae >0$ was arbitrary it follows that $x \in K(t) \subset \bigcup_{t \in [0,T]} K(t)$.

\end{proof}

\section{The space $(\M,d_{\M})$} \label{METRIC}
We let $\P$ denote the set of all measures with compact support in $X$ and
$$\M =\{\nu \in \P: 0 \leq \nu \leq \mu\}.$$
In particular $ 0 \in \M$.
We note that $\M$ is locally closed for the weak$^*$-topology in the sense that any sequence $\eta_i \in \M$ such that $\eta_i \rightharpoonup \eta$ in $\P$ and $\bigcup_{i=1}^{\infty} \supp(\eta_i)$ is contained in a compact subset of $X$, then $\eta$ also belongs to $\M$. 
\begin{remark}
A measure $\nu$ belongs to $\M$ if and only if there is a measurable function $\phi : X \rightarrow [0,1]$ with compact support such that $\nu = \phi \mu$.
Also note that the lattice operations are equivalent in the following sense if $\nu_i=\phi_i \mu$:
$$\nu_1 \vee \nu_2 = (\phi_1 \vee \phi_2) \mu, \quad \nu_1 \wedge \nu_2 = (\phi_1 \wedge \phi_2) \mu.$$
Furthermore $\|\nu\|=\|\phi\|_{L^1(\mu)}$ and if $\nu_i = \phi_i \mu$ and $\nu=\phi \mu$ belongs to $\M$,  then $\nu_i$ converges weak$^*$ to $\nu$ in $\M$ if and only if $\phi_i$ converges weak$^*$ to $\phi$ in $L^{\infty}(\mu)$.
Hence one could alternatively think of the elements in $\M$ as consisting of all such functions $\phi$ rather than measures with essentially no changes in the proofs below. 
\end{remark} 
We will now introduce a metric $d_{\M}$ on $\M$. To do this we first fix a strictly increasing continuous function $h: [0,\infty) \rightarrow [0,\infty)$ such that
\begin{align*}
& h(0) =\lim_{\vae \rightarrow 0^+} \frac{h(\vae)}{\vae} =0,\\
& h(\vae_1) + h(\vae_2) \leq h(\vae_1 + \vae_2) \quad \forall \vae_1,\vae_2 \in [0,\infty),\\
& \lim_{\vae \rightarrow \infty} h(\vae)=\infty.
\end{align*}
One example of $h$ is $h(\vae)=\vae^s$ for any fixed $s \in (1,\infty)$.
The construction of the metric depends on decompositions of measures, and it will be convenient to introduce for $\eta$ and $\nu$ in $\M$ and $\vae,\delta >0$
$$\Gamma_{\vae,\delta}(\nu,\eta) = \left\{ (\nu_i,\eta_i)_{i \in \MM}: \begin{array}{l} 
\MM \textrm{ is at most countable},\\
\nu_i,\eta_i \in \M \textrm{ for each } i \in \MM\\
\nu = \sum_{i \in \MM} \nu_i,\quad \eta = \sum_{i \in \MM}\eta_i,\\
\sum_{i\in \MM}  \bigl|\|\nu_i\|-\|\eta_i\|\bigr| \leq \delta,\\
\textrm{diam}(\supp(\nu_i) \cup \supp(\eta_i)) \leq \vae. \end{array}\right\}.$$
We also introduce 
$$\Gamma_{\vae}(\nu,\eta)=\Gamma_{\vae,h(\vae)}(\nu,\eta).$$
To make the notation less cumbersome we will often drop the index set when it is clear from the context and simply write $(\nu_i,\eta_i) \in \Gamma_{\vae,\delta}(\nu,\eta)$.
\begin{remark}
Of-course we could in the definition above have worked with only $\N$ instead of a general set $\MM$, but this is for convenience later, since we often will have for instance double subscripts, and we wish to avoid the need to relabel these.
\end{remark}

We now introduce a metric on $\M$ as follows:
\begin{equation}\label{mmetric}
d_{\M}(\nu,\eta) := \inf\left\{ \vae : \Gamma_{\vae}(\nu,\eta) \ne \emptyset \right\}.
\end{equation}
\begin{remark}
Note that if we put $\eta_1=\eta$, $\nu_1=\nu$ and $\eta_i=\nu_i=0$ for $i \ne 1$, and simply choose $\vae>0$ large enough such that $\bigl| \|\eta\|-\|\nu\| \bigr| < h(\vae)$ (which is possible since $h(\vae) \rightarrow \infty$ as $\vae \rightarrow \infty$) and ${\rm diam}(\supp(\nu) \cup \supp(\eta)) < \vae$ then $(\nu_i,\eta_i)\in \Gamma_{\vae}(\nu,\eta)$ and hence $d_{\M}(\nu,\eta) \leq \vae$. So $d_{\M}(\nu,\eta)$ is always finite. 
\end{remark}
\begin{remark}\label{metrem}
It is easy to see that we could just as well have restricted ourselves to finite sums rather than countable ones in the definition of $\Gamma_{\vae,\delta}$ without altering the metric $d_{\M}$, but allowing countable sums makes it easier to work with.

The choice of $h$ of-course makes a difference for the metric in the sense in how expensive it is to enlarge the mass, but the particular choice of $h$ will not be very important to us as we will see, because we will work mainly with rectifiable curves, and the role of $h$ then just becomes to force the total mass of the measures along such a curve to be constant (which will always be the case as long as $h$ satisfies the assumptions above). 
\end{remark}
Here are some simple consequences of the definition:
\begin{theorem}
Suppose $d_{\M}(\nu,\eta) \leq \delta$, then
\begin{itemize}
\item[(1)] $\bigl| \|\nu\|-\|\eta\| \bigr| \leq h(\delta)$,
\item[(2)] $\eta(({\rm supp}(\nu)_{\delta})^c) \leq h(\delta)$.
\end{itemize}
\end{theorem}
\begin{proof}
Suppose $\delta < \vae$. Then by definition there is an element $(\nu_i,\eta_i) \in \Gamma_{\vae}(\nu,\eta)$, and by definition this means that
$$\begin{array}{l} 
\nu = \sum_{i=1}^ {\infty} \nu_i,\quad \eta = \sum_{i=1}^{\infty} \eta_i,\\
\sum_{i=1}^{\infty}  \bigl| \| \nu_i \|- \|\eta_i \| \bigr| < h(\vae),\\
\textrm{diam}(\supp(\nu_i) \cup \supp(\eta_i)) <\vae. \end{array}$$
Therefore
$$\bigl| \|\nu \|- \| \eta \| \bigr| \leq \sum_{i=1}^{\infty}  \bigl| \| \nu_i \|- \|\eta_i \| \bigr| < h(\vae).$$
By continuity of $h$ we get that 
$$\bigl| \| \nu \| - \| \eta \| \bigr|  \leq h(\delta).$$

Now let $I$ denote the set of all $i$ such that $\nu_i \ne 0$ above.
Then it is clear that ${\rm supp}(\eta_i) \subset {\rm supp}(\nu_i)_{\vae} \subset {\rm supp}(\nu)_{\vae}$. Hence
$$\eta(({\rm supp}(\nu)_{\vae})^c) \leq \sum_{i \not \in I} \|\eta_i\|\leq h(\vae).$$
Again by continuity of $h$ and the fact that
$({\rm supp}(\nu)_{\delta + 1/n})^c$ increases to $({\rm supp}(\nu)_{\delta})^c$ as $n \rightarrow \infty$ we get the desired estimate.
\end{proof}
\begin{lemma}\label{metlem}
Suppose $\eta,\nu,\gamma_{\eta},\gamma_{\nu} \in \M$ are such that also $\nu+\gamma_{\nu}, \eta+\gamma_{\eta} \in \M$. If $\|\gamma_{\nu} + \gamma_{\eta}\| \leq \delta$ and  
 $\Gamma_{\vae,\delta - ||\gamma_{\nu} + \gamma_{\eta}||}(\nu,\eta) \ne \emptyset$,
then $\Gamma_{\vae,\delta}(\nu+\gamma_{\nu},\eta+\gamma_{\eta}w) \ne \emptyset$. 
\end{lemma}
\begin{proof}Suppose $(\nu_i,\eta_i) \in \Gamma_{\vae,\delta - ||\gamma_{\nu} + \gamma_{\eta}||}(\nu,\eta)$.
Now cover the support of $\gamma_{\nu} + \gamma_{\eta}$ by finitely many balls $B_1,B_2,\ldots,B_k$ of radius at most $\vae/2$, let 
$$\nu_i'= \left\{\begin{array}{ll} \gamma_{\nu}|_{B_1} & i=1,\\ \gamma_{\nu}|_{B_i \setminus \bigcup_{j=1}^{i-1}B_j} & 1<i \leq k \\ \nu_{i-k} & i >k,\end{array}\right.$$
and
$$\eta_i'= \left\{\begin{array}{ll} \gamma_{\eta}|_{B_1} & i=1,\\ \gamma_{\eta}|_{B_i \setminus \bigcup_{j=1}^{i-1}B_j} & 1<i \leq k \\ \eta_{i-k} & i >k,\end{array}\right.$$
Then it is straightforward to show that $(\nu_i',\eta_i') \in \Gamma_{\vae,\delta}(\nu+\gamma_{\nu},\eta+\gamma_{\eta})$.
\end{proof}

The following lemma will be first of all used to prove the triangle inequality for $d_{\M}$.
\begin{lemma}\label{split}
Suppose $\eta,\nu$ belong to $\M$, $\vae,\delta>0$ and that $(\nu_j',\eta_j') \in \Gamma_{\vae,\delta}(\nu,\eta)$. Suppose also that $\eta=\sum_{i=1}^{\infty} \eta_i$ where each $\eta_i$ belongs to $\M$, and let $\delta_j = \left|\| \nu_j'\|-\|\eta_j'\|\right|$.

Then there is for each $j \in \N$ 
$$(\nu_{i,j},\eta_{i,j})_{ i \in \N} \in \Gamma_{\vae,\delta_j}(\nu_j',\eta_j')$$
with the additional property that for each $i \in \N$
$$\eta_i=\sum_{j=1}^{\infty} \eta_{i,j}.$$

Furthermore, if we define $\nu_i = \sum_{j=1}^{\infty} \nu_{i,j}$ and let $\rho_i = \sum_{j=1}^{\infty}\left| \| \nu_{i,j}\| -\|\eta_{i,j}\|\right|$, then 
$$(\nu_{i,j},\eta_{i,j})_{j \in \N} \in \Gamma_{\vae,\rho_i}(\nu_i,\eta_i) \textrm{ for each } i \in \N.$$

Finally we have
$$(\nu_{i,j},\eta_{i,j})_{i,j \in \N} \in \Gamma_{\vae,\delta}(\nu,\eta).$$
\end{lemma}

\begin{proof}
If we let $\phi,\phi_i$ and $\phi_j'$ denote the densities of $\eta,\eta_i$ and $\eta_j'$ with respect to $\mu$ respectively and define $\eta_{i,j} = \phi_{i,j}\mu$ where
$$\phi_{i,j}(x)=\left\{\begin{array}{ll} \frac{\phi_i(x)\phi_j'(x)}{\phi(x)} & \phi(x)>0,\\ 0 & \phi(x)=0,\end{array}\right.$$
then 
$$\eta_i = \sum_{j=1}^{\infty} \eta_{i,j} \quad \textrm{and} \quad \eta_j' = \sum_ {i=1}^{\infty} \eta_{i,j}.$$ 
We have $\|\eta_j'\| - \delta_j \leq \|\nu_j'\|\leq \|\eta_j'\| + \delta_j$, so $\sum_{j=1}^{\infty} \delta_j \leq \delta$, and we will now divide each $\nu_j'$ into pieces $\nu_{i,j}$ such that  
$$\nu_j'=\sum_{i=1}^{\infty} \nu_{i,j} \quad \textrm{and}  \quad \sum_{i=1}^{\infty} \bigl| \| \nu_{i,j} \| - \| \eta_{i,j} \| \bigr| \leq \delta_j,$$
and then verify that these measures satisfies the other required properties of the lemma.

Let $I=\{j \in \N: \eta_j'=0\}$. If $j \in I$ then we have $\eta_{i,j}=0$ as well for all $i$, and we simply let $\nu_{1,j}=\nu_j'$ and $\nu_{i,j}=0$ for $i>1$. 

In case $j \in \N \setminus I$ then put
$$\alpha_j=\|\nu_j'\|/\|\eta_j'\|$$
and
$$\nu_{i,j}=\frac{\|\eta_{i,j}\|}{\|\eta_j'\|}\nu_j'.$$
Since $\nu_{i,j} \leq \nu_j'$ it belongs to $\M$. It is also clear that
$$\sum_{i=1}^{\infty} \nu_{i,j} = \left(\sum_{i=1}^{\infty} \frac{\|\eta_{i,j}\|}{\|\eta_j'\|}\right)\nu_j' = \nu_j'. $$
 
Since for any $i,j \in \N$ we have $\supp(\eta_{i,j}) \subset \supp(\eta_j')$ and $\supp(\nu_{i,j}) \subset \supp(\nu_j')$
it is clear that ${\rm diam}(\supp(\eta_{i,j}) \cup \supp(\nu_{i,j})) < \vae$.

For $j \in \N \setminus I$ we have
$$|\alpha_j-1| =\left|\frac{\|\nu_j'\|}{\|\eta_j'\|}-1\right| = \frac{\delta_j}{\|\eta_j'\|}.$$
In case $j \in I$, then 
$$\sum_{i=1}^{\infty} \left| \|\nu_{i,j}\| - \|\eta_{i,j}\|\right| = \|\nu_j'\| = \left| \|\nu_j'\| - \|\eta_j'\|\right| = \delta_j.$$
In case $j \in \N \setminus I,$ then
$$\sum_{i=1}^{\infty} \left| \|\nu_{i,j}\| - \|\eta_{i,j}\|\right| = \sum_{i=1}^{\infty} \left| (\alpha_j - 1)\|\eta_{i,j}\|\right| = |\alpha_j -1 |\|\eta_j'\| = \left| \|\nu_j'\|- \|\eta_j'\|\right| = \delta_j.$$
Hence we see that
$$(\nu_{i,j},\eta_{i,j})_{i \in \N} \in \Gamma_{\vae,\delta_j}(\nu_j',\eta_j').$$

The second statement also follows trivially by definition from the above statement about the supports of the measures. To prove the final claim we have
\begin{align*}
&\sum_{i=1}^{\infty}\bigl| \|\nu_i\|-\|\eta_i\|\bigr|= \sum_{i=1}^{\infty}\left| \sum_{j=1}^{\infty}\|\nu_{i,j}\| - \sum_{j=1}^{\infty} \|\eta_{i,j}\|\right| \leq \sum_{j=1}^{\infty} \sum_{i=1}^{\infty}\bigl| \| \nu_{i,j} \| - \| \eta_{i,j} \| \bigr|\\
&= \sum_{j=1}^{\infty} \delta_j \leq \delta. 
\end{align*}
\end{proof}

\begin{theorem}
$(\M,d_{\M})$ is a metric space.
\end{theorem}
\begin{proof}
To prove that $d_{\M}(\eta,\eta)=0$ for every $\eta \in \M$ let $\vae >0$ and cover $\supp(\eta)$ by finitely many balls 
$$B_1,B_2,\ldots,B_k$$
with radius at most $\varepsilon/2$, put 
$$\eta_i = \left\{\begin{array}{ll} \eta|_{B_1} & i=1 \\ \eta|_{B_i \setminus (\bigcup_{j=1}^{i-1}B_i)} & 1<i \leq k \\ 0 & i>k\end{array}\right..$$
Then $(\eta_i,\eta_i) \in \Gamma_{\vae}(\eta,\eta)$ so $d_{\M}(\eta,\eta) \leq \varepsilon$.

To prove that $d_{\M}(\eta,\nu)=0 \Rightarrow \eta=\nu$ it is enough to show that $\int f \, d \eta = \int f \, d\nu$ for every uniformly continuous function $f$ with values in $[0,1]$. So given $\varepsilon > 0$ we may choose $\delta \in (0,\vae)$ such that
$$\inf_{B(x,\delta)}f \geq \sup_{B(x,\delta)} f -\varepsilon \quad \forall x \in X.$$
Then we may by assumption choose $(\nu_i,\eta_i) \in \Gamma_{\delta}(\nu,\eta)$ and we get
 with $A_i =\supp(\nu_i)\cup \supp(\eta_i)$
\begin{align*}
&\left|\int f \, d \nu - \int f \, d\eta \right| =\left| \sum_{i=1}^{\infty} \left(\int f d \nu_i- \int f \, d\eta_i\right)\right| \\
&=\left|\sum_{i=1}^{\infty} \left(\sup_{A_i} f\|\nu_i\| + \int (f - \sup_{A_i}f) d \nu_i -\sup_{A_i}f\|\eta_i\|- \int (f-\sup_{A_i} f ) \, d\eta_i\right)\right|\\
&\leq\sum_{i=1}^{\infty} \sup_{A_i}f \bigl| \| \nu_i \| - \| \eta_i \|\bigr| + \sum_{i=1}^{\infty} \varepsilon (\|\nu_i\|+\|\eta_i\|) \\
&\leq\sum_{i=1}^{\infty} \bigl| \| \nu_i \| - \| \eta_i \| \bigr| + \varepsilon(\|\nu|+\|\eta\|) \leq h(\delta) +  \varepsilon(\|\nu|+\|\eta\|) \leq  h(\vae) +  \varepsilon(\|\nu|+\|\eta\|).
\end{align*} 
That $d_{\M}(\nu,\eta)=d_{\M}(\eta,\nu)$ is obvious, so it remains to prove the triangle inequality. Suppose therefore that $\rho,\xi,\tau \in \M$ with $$d_{\M}(\rho,\xi) < \vae_1 \quad \textrm{and} \quad d_{\M}(\xi,\tau) < \vae_2.$$
By the definition of the metric $d_{\M}$ there are
\begin{align*}
(\rho_i,\xi_i) \in \Gamma_{\vae_1}(\rho,\xi),\\
(\xi_j',\tau_j') \in \Gamma_{\vae_2}(\xi,\tau).
\end{align*}
By Lemma \ref{split} applied to $\nu = \tau$, $\eta=\xi$, $\eta_i = \xi_i$ and  $\delta_j = \left|\|\xi_j'\| - \|\tau_j'\|\right|$ we may now find
$$(\xi_{i,j},\tau_{i,j})\in \Gamma_{\vae_2,\delta_j}(\xi_j',\tau_j') \textrm{ for all } j \in \N$$
such that
$$\xi_i = \sum_{j=1}^{\infty} \xi_{i,j} \textrm{ for all } i \in \N.$$

We will now apply Lemma \ref{split} again, but this time for each $k$ applied to $\nu=\rho_k$, $\eta=\xi_k$, $\eta_i= \xi_{k,i}$ so that $\eta = \sum_{i=1}^{\infty} \eta_i$. If we let
$$\nu_j' = \left\{ \begin{array}{ll} \nu & j=1\\0 & j >1\end{array}\right., \quad \eta_j' = \left\{ \begin{array}{ll} \eta & j=1\\0 & j >1\end{array}\right.$$
and
$$\delta_j = \left|\|\nu_j'\|- \|\eta_j'\|\right|,$$
then $(\nu_j',\eta_j')_{j \in \N} \in \Gamma_{\vae_1,\delta_1}(\rho_i,\xi_i)$ by definition. Now we let $\rho_{k,i} =\nu_i$, where $\nu_i$ is as in Lemma \ref{split} and we get, since $\nu_i \leq \nu$,  
$${\rm diam}(\supp(\rho_{i,j}) \cup \supp(\xi_{i,j})) \leq {\rm diam}(\supp(\rho_i) \cup \supp(\xi_i)) \leq \vae_1 \textrm{ for all } j \in \N$$
and
$$\sum_{i,j=1}^{\infty} \bigl| \| \rho_{i,j} \| - \| \xi_{i,j} \| \bigr| \leq h(\vae_1).$$ 
Furthermore by construction for every $i,j$ such that $\xi_{i,j} \ne 0$ we have  
$$\textrm{diam}(\supp(\rho_{i,j}) \cup \supp(\tau_{i,j})) \leq \vae_1 + \vae_2.$$
If we let $I=\{i,j:  \xi_{i,j} \ne 0\}$ then we get
\begin{align*}
&\sum_{i,j \in I}  \bigl| \| \rho_{i,j} \| - \| \tau_{i,j} \| \bigr| + \sum_{i,j \not \in I} \|\rho_{i,j}\|+\sum_{i,j \not \in I} \|\tau_{i,j}\|  \\
&\leq \sum_{i,j \in I} \bigl| \|\rho_{i,j} \| - \| \xi_{i,j} \|\bigr| + \sum_{i,j \not \in I} \|\rho_{i,j}\| + \sum_{i,j \in I}\bigl| \| \xi_{i,j} \| - \| \tau_{i,j} \| \bigr| +\sum_{i,j \not \in I} \|\tau_{i,j}\|  \\
&\leq h(\vae_1)+h(\vae_2) \leq h(\vae_1+\vae_2).
\end{align*}
From this it follows from Lemma \ref{metlem} that indeed 
$d_{\M}(\rho,\tau) < \vae_1+\vae_2$, so we have proved the triangle inequality.
\end{proof}
\begin{proposition}
Suppose $K \subset X$ is compact, and that the measures $\eta_i,\eta \in \M$ where all $\eta_i$ have support in $K$. Then $\eta_i \rightarrow  \eta$ in $\M$ if and only if $\eta_i \rightharpoonup \eta$ weak$^*$. In particular the set 
$$\K = \{\eta \in \M : {\rm supp}(\eta) \subset K\}$$
is a compact subset of $\M$.
\end{proposition}
\begin{remark}
In particular, in case $X$ is compact then so is $\M$, and hence it is complete.
In case $X$ is not compact, then the space $(\M,d_{\M})$ is not even complete. To explain why let $0 \leq \phi(x) \leq 1$ for all $x \in X$ and $\int \phi \, d \mu < \infty$, but such that $\phi$ does not have compact support. Then we may define $\eta_n= \phi \mu|_{B(0,n)}$. Each $\eta_n$ belongs to $\M$, and it is easy to see that it is a Cauchy sequence in $\M$. But of-course it does not converge to an element in  $\M$. This is in a sense the price we pay to require that all our elements in $\M$ should have compact support. However as we will see in the next section this is not an issue for rectifiable curves, and hence this will not be an actual problem for us.

Furthermore note that since convergence in $\M$ by the above implies weak$^*$ convergence of the densities in $L^{\infty}(\mu)$ this implies that if $\eta_i \rightarrow \eta$ in  $\M$, then 
$$\int f \,d\eta_i \rightarrow \int f\, d\eta \textrm{ for every } f \in L^1_{\rm loc} (X).$$  
\end{remark}
\begin{proof}
Assume that $\eta_i$ converges to $\eta=\eta_{\infty}$ in the weak$^*$-topology.
We will now prove that $\eta_i$ converges to the measure $\eta_{\infty}$ in $d_{\M}$.
Given $\vae >0$ we may cover $K$ by finitely many balls $B(x_1,\vae/2),B(x_2,\vae/2),\ldots,B(x_k,\vae/2)$. 
Choose a partition of unity $f_1,f_2,\ldots,f_k$ of continuous functions such that 
$$\sum_{i=1}^k f_i(x) =1 \quad \forall x \in \bigcup_{i=1}^k B(x_i,\vae/2),$$
$$0 \leq f_i \leq 1 \quad \forall i\in \{1,2,\ldots,k\}$$
and
$$\supp(f_i) \subset B(x_i,\vae) \quad \forall i\in \{1,2,\ldots,k\}.$$ 
Now we define for each $i,j$ the measures $\eta_{i,j} = f_j \eta_i$, and conclude that 
$$d_{\M}(\eta_{\infty},\eta_i) \leq \max \left\{ h^{-1}\left({\sum_{j=1}^k  \bigl| \|\eta_{i,j}\| - \| \eta_{\infty,j} \|\bigr|}\right), \vae \right\},$$
and since the first factor goes to zero as $i \rightarrow \infty$ we get the statement.

In case $\eta_i$ converges to $\eta$ in $(\M,d_{\M})$ and $\supp(\eta_i) \subset K$ for each $i$, then for any $\vae >0$ we get $\eta(K^c) < h(\vae),$ and hence $\eta(K^c)=0$. So $\supp(\eta) \subset K$. Furthermore, by the above argument, if a subsequence converges in the weak$^*$ topology, then the limit must be $\eta$, and hence we also get the opposite direction. (Note also that $\{\nu \in \M: \supp(\nu) \subset K\}$ forms a compact subset under the weak$^*$ topology, and hence any sequence in this set contains a convergent subsequence.)
\end{proof}
\begin{proposition}\label{splitcor}
Suppose $\eta^1,\eta^2,\ldots,\eta^N$ belongs to $\M$ and satisfies $d_{\M}(\eta^{k+1},\eta^k) \leq \vae_k$ for each $k =1,\ldots,N-1$. Let  $r \in \{1,2,\ldots,N\}$ and suppose that  
$$\eta^r = \sum_{i=1}^{\infty} \gamma_i^r$$
where each $\gamma_i^r \in \M$. Then there are measures $\gamma_i^k$ in $\M$, $k \in  \{1,2,\ldots,N\} \setminus \{r\}$, such that  
\begin{itemize}
\item[(1)] $\eta^k = \sum_{i=1}^{\infty} \gamma_i^k \textrm{ for all } k \in \{1,2, \ldots,N\}$,
\item[(2)] $d_{\M}(\gamma_i^{k+1},\gamma_i^k) \leq \vae_k \textrm{ for all } k \in \{1,2,\ldots,N-1\},\, i \in \N$
\item[(3)] $\sum_{i=1}^{\infty}\bigl| \|\gamma_i^{k+1}\|- \| \gamma_i^k \| \bigr| \leq h(\vae_k) \textrm{ for all } k \in \{1,2,\ldots,N-1\}.$
\end{itemize}
\end{proposition}
\begin{remark}
In case $N=2$ and we have strict inequalities this is a special case of Lemma \ref{split}.
\end{remark}
\begin{proof}
It is enough to prove this for the case $N=2$, since then we may simply iterate this result. By symmetry in this case we can also without loss assume that $r=1$.
Let $n \in N$ and choose
$$(\nu_j',\eta_j') \in \Gamma_{\vae_1 +1/n}(\nu,\eta).$$ 
Now apply Lemma \ref{split} to 
\begin{align*}
& \eta=\eta^1,\\
& \nu=\eta^2,\\
& \eta_i = \left\{ \begin{array}{ll} \gamma_1^1 & i=1 \\ \sum_{j=2}^{\infty}\gamma_j^1 & i=2 \\ 0 & i>2\end{array}\right.,
\end{align*}
to get measures
$\gamma_1^2(n),\nu^2(n)$ (with the notation of Lemma \ref{split} $\gamma_1^2(n) = \nu_1$ and $\nu^2(n) = \nu-\nu_1$) such that $\gamma_1^2(n)+\nu^2(n)=\eta^2$ and:
\begin{align*}
&d_{\M}(\gamma_1^1,\gamma_1^2(n)) <\vae_1+1/n,\\
&d_{\M}(\nu^1,\nu^2(n)) <\vae_1+1/n,\\
&\bigl| \|\gamma_1^1\|-\|\gamma_1^2(n)\|\bigr| + \bigl|\|\nu^1\|-\|\nu^2(n)\|\bigr| < h(\vae_1 +1/n).
\end{align*}
If we do this for each $n$ we get a sequence of measures, and since the measures $\gamma_1^2(n)$ and $\nu^2(n)$ all have support in the compact set ${\rm supp}(\eta^2)$ it follows that there is a sequence $n_1,n_2,\ldots$ such that both $\gamma_1^2(n_j) \rightarrow \gamma_1^2$ and $\nu^2(n_j) \rightarrow \nu^2$ as $j \rightarrow \infty$ for some measures $\gamma_1^2,\nu^2 \in \M$. It is clear that we still have $\gamma_1^2 + \nu^2 = \eta^2$, and that by construction
\begin{align*}
&d_{\M}(\gamma_1^1,\gamma_1^2) \leq \vae_1,\\
&d_{\M}(\nu^1,\nu^2) \leq \vae_1,\\
&\bigl|\|\gamma_1^1\|-\|\gamma_1^2(n)\|\bigr| + \bigl|\|\nu^1\|-\|\nu^2(n)\|\bigr| \leq h(\vae_1).
\end{align*}
In the next step we may apply the same construction to the measures $\nu^1$ and $\nu^2$, but this time within the class $\Gamma_{\vae_1,\delta}(\nu^1,\nu^2)$ where $\delta=h(\vae_1) - \bigl|\|\gamma_1^1\|-\|\gamma_1^2\|\bigr|$, to get our measure $\gamma_2^2$, and iterating this leads to a sequence of measures  $\gamma_i^2$ with the properties that for each $j$
\begin{align*}
&d_{\M}(\gamma_j^1,\gamma_j^2) \leq \vae_1,\\
& \sum_{i=1}^{j} \bigl| \|\gamma_i^1\|-\|\gamma_i^2\|\bigr| + \left| \bigl\|\eta^1-\sum_{i=1}^j \gamma_i^1\bigr\|-\bigl\|\eta^2 - \sum_{i=1}^j \gamma_i^2\bigr\|\right| \leq h(\vae_1).
\end{align*}
Hence we see that $\rho=\eta^2-\sum_{i=1}^{\infty} \gamma_i^2$ is an element in $\M$ with total mass not bigger than $h(\vae_1)$. In case it is not zero, we may by Lemma \ref{metlem} simply add it to any of the measures, say $\gamma_1^2$, and we get the required measures.

\end{proof}
An important principle for us will be how one can estimate distances in $\M$ in case one measure is given from another one trough a measure preserving map as follows.
\begin{theorem}\label{Hmap1}
Suppose the map $H : X \rightarrow X$ is a homeomorphism such that for any compact subset $K$ of $X$ we have $\mu(K)=c\mu(H^{-1}(K))$ where $c \in (0,1]$ is fixed, and that there is a number $t \geq 0$ such that 
$$d_{X}(x,H^{-1}(x)) \leq t \textrm{ for all } x \in X.$$
Suppose furthermore that $\phi : X \rightarrow [0,1]$ is measurable and that $({\rm supp}(\phi))_t$ is a compact subset of $X$. Then $c(\phi \circ H) \mu, \phi \mu \in \M$ and 
$$d_{\M}(c(\phi \circ H) \mu, \phi \mu ) \leq t.$$
\end{theorem}
\begin{remark}
Obviously (since $H$ is a homeomorphism) the condition $\mu(K)=c\mu(H^{-1}(K))$ is equivalent to $\mu(H(K))=c\mu(K)$ and $d_{X}(x,H^{-1}(x)) \leq t \textrm{ for all } x \in X$ is equivalent to $d_{X}(x,H(x)) \leq t \textrm{ for all } x \in X.$
\end{remark}
\begin{proof}
Let $\vae >0$ and choose $\delta \in (0,\vae)$ such that 
$$d_X(H^{-1}(x),H^{-1}(y)) \leq \vae \textrm{ for all } x,y \in ({\rm supp}(\phi))_t \textrm{ such that } d_X(x,y) \leq \delta.$$

We may write ${\rm supp}(\phi) = A_1 \cup A_2 \cup \cdots \cup A_k$, where the sets $A_i$ are disjoint and measurable with diameter at most $\delta$.
Let 
$$\eta_i=c((\phi\chi_{A_i}) \circ H)\mu \textrm{ and } \nu_i = (\phi\chi_{A_i})\mu.$$
Then 
$$c(\phi \circ H) \mu = \sum_{i=1}^{k} \eta_i \textrm{ and } \phi \mu = \sum_{i=1}^k \nu_i.$$ 
Since it follows from the assumptions on $H$ that 
$$c \int f \circ H \, d\mu = \int f \, d\mu \textrm{ for all } f \in L^1(X),$$
we see that $\|\nu_i\|=\|\eta_i\|$.
Furthermore, since ${\rm supp}(\nu_i) \subset A_i$ and $ {\rm supp}(\eta_i) \subset H^{-1}(A_i)$, we have
$${\rm diam}({\rm supp}(\nu_i) \cup {\rm supp}(\eta_i)) \leq t+\vae +\delta<t+2\vae.$$
Hence 
$$d_{\M}(c(\phi \circ H) \mu,\phi \mu) \leq t+2\vae.$$
Since $\vae >0$ was arbitrary the result follows.
\end{proof}
The next theorem will not really be useful to us since it concerns non-rectifiable curves, but it explains a bit of the nature of the metric space $\M$. (In particular we should note that a very natural type of curve will typically be non-rectifiable with our metric $d_{\M}$.)
\begin{theorem}
$\M$ is path-wise connected. Indeed if $\eta,\nu \in \M$ then $(1-t)\eta + t \nu$, $0\leq t \leq1$, is a (typically non-rectifiable) path connecting $\eta$ to $\nu$. Furthermore $d_{\M}(\eta,(1-t)\eta + t \nu) \leq d_{\M}(\eta,\nu)$.
Hence $\M$ is also path-wise locally connected.
\end{theorem}
\begin{proof}
Let $\vae > d_{\M}(\eta,\nu)$. We may then choose $(\nu_i,\eta_i) \in \Gamma_{\vae}(\nu,\eta)$. 
Now we split 
$$(1-t)\eta + t \nu = \sum_{i=1}^{\infty} (1-t) \eta_i + \sum_{i=1}^{\infty} t\nu_i,$$
and
$$\eta = \sum_{i=1}^{\infty} (1-t)\eta_i + \sum_{i=1}^{\infty}t\eta_i.$$
If we apply the definition of the metric to these decompositions of the measures we see that
indeed the diameters of the unions of the supports are unchanged, and
\begin{align*}
&\sum_{i=1}^{\infty}\bigl|\|(1-t)\eta_i\|-\|(1-t)\eta_i\|\bigr| + \sum_{i=1}^{\infty} \bigl|\|t\eta_i\|-\|t\nu_i\|\bigr| \\
&=t  \sum_{i=1}^{\infty} \bigl|\|\eta_i\|-\|\nu_i\|\bigr| < t h(\vae) <h(\vae).
\end{align*}
Hence we see that $d_{\M}(\nu,(1-t)\nu+t\eta)<\vae.$
Now we may also do a similar argument to $(1-s)\eta+s\nu =\sum_{i=1}^{\infty} (1-s)\eta_i+ \sum_{i=1}^{\infty} s\nu_i$ and $(1-t)\eta+t\nu=\sum_{i=1}^{\infty} (1-t)\eta_i+ \sum_{i=1}^{\infty} t\nu_i$ and note that since 
$$\sum_{i=1}^{\infty} \bigl|\|(1-s)\eta_i\|-\|(1-t)\eta_i\|\bigr| + \sum_{i=1}^{\infty} \bigl|\|s\nu_i\|-\|t\nu_i\|\bigr| = 
|s-t|(\|\eta\|+\|\nu\|),$$
it follows that the curve is continuous as stated.
\end{proof}
\subsection{Rectifiable curves in $\M$}
Rectifiable curves will play a crucial role for us in our construction of Sobolev type spaces. Both of the results in the first theorem are rather direct consequences of our definitions, but they will be important to us later. 
\begin{theorem}\label{rectcont}
If $\eta \in \WR(\M)$ then
\begin{itemize}
\item[(1)] $\|\eta(t)\|$ is constant,
\item[(2)] If $s,t \in [0,b_{\eta}]$ and $|s-t| < \vae$, then ${\rm supp}(\eta(t)) \subset ({\rm supp}(\eta(s)))_{\vae}$. 
\end{itemize}
\end{theorem}
\begin{remark}
Note in particular that part (2) implies, according to Lemma \ref{K_t}, that for a given curve $\eta \in \WR(\M)$ the set $\bigcup_{s \in [0,b_{\eta}]} {\rm supp}(\eta(s))$ is compact.  This is what we meant by that the non-completeness of the space (in case $X$ is not compact) is not an actual problem for rectifiable curves, since we have control of the supports.
\end{remark}
\begin{proof}
\noindent(1): This follows from Lemma \ref{fundth} since
$$\limsup_{s \rightarrow t} \left| \frac{\|\eta(s)\|-\|\eta(t)\|}{s-t}\right| \leq \limsup_{s \rightarrow t} \frac{h(|s-t|)}{|s-t|} =0.$$

\noindent (2): Let $t$ be fixed. It is enough to consider the case $s=0$, $\vae =t$ (by time reversal and/or translation if necessary). To do so let $A={\rm supp}(\eta(0))$ and we will prove that for any  given $N \in \N$ and any $k\in\{1,2,\ldots,N\}$ we have that
\begin{equation}\label{est1}
\eta(kt/N)((A_{kt/N})^c) \leq (2k-1)h(t/N).
\end{equation}
From this the result follows, since then
$$\eta(t)(A_t^c) =\eta(Nt/N)((A_{Nt/N})^c) \leq (2N-1)h(t/N) =\frac{t(2N-1)}{N}\frac{h(t/N)}{t/N},$$
which by the assumption on $h$ goes to zero as $N \rightarrow \infty$.
The case $k=1$ is simply by definition since $d_{\M}(\eta(0),\eta(t/N)) \leq t/N$.
Assume now that formula (\ref{est1}) is true for all $k < k_0+1$. Then we may write
$$\eta(k_0t/N) = \eta_1+\eta_2,$$
where
$$\eta_1=\eta(k_0t/N)|_{A_{k_0t/N}}, \quad \eta_2= \eta(k_0t/N)|_{A_{k_0t/N}^c}.$$
We may then according to Proposition \ref{splitcor} write $\eta((k_0+1)t/N) = \eta_1'+\eta_2'$
where $d_{\M}(\eta_i,\eta_i') \leq t/N$ for $i=1,2$.
In particular $\|\eta_2'\| \leq h(t/N) +\|\eta_2\| \leq 2k_0h(t/N)$. 

Also
$$\eta_1'(({\rm supp}(\eta_1)_{t/N})^c) \leq h(t/N),$$
and since 
$${\rm supp}(\eta_1)_{t/N} \subset (A_{k_0t/N})_{t/N} \subset A_{(k_0+1)t/N}$$ we get that 
$$\eta_1'(A_{(k_0+1)t/N}^c) \leq h(t/N).$$
Summing up we get
$$\eta((k_0+1)t/N)(A_{(k_0+1)t/N}^c) \leq (2k_0 + 1)h(t/N) = (2(k_0+1)-1)h(t/N),$$
and the proof is done.
\end{proof}
The following is a fundamental adaptation of Proposition \ref{splitcor} to rectifiable curves.
\begin{proposition}\label{split2}
Suppose that $\MM$ is at most countable, $\eta \in \WR(\M)$, $t \in [0,b_{\eta}]$ and $\gamma_i \in \M$ for each $i\in \MM$ are such that
$\eta(t)= \sum_{i \in \MM} \gamma_i$. Then there are curves $\eta_i : [0,b_{\eta}]  \rightarrow \M$ in $\WR(\M)$ such that
\begin{itemize}
\item[(1)] $\eta_i(t)=\gamma_i$ for each $i \in \MM$,
\item[(2)] $\eta(s)= \sum_{i \in \MM} \eta_i(s)$ for each $s \in [0,b_{\eta}]$.
\end{itemize}
\end{proposition}
\begin{remark}
Note that the curves $\eta_i$ has length at most $b_{\eta}$ but there are certainly situations where, for a particular $i$, the curve may have strictly smaller length even if $\eta \in \RR(\M)$. For example, suppose $\eta_1 \in \RR(\M)$ and $\eta_2 \in \M$ are such that the supports of $\eta_1(s)$ and $\eta_2$ are separated from each other for all $s$, and define $\eta(s)=\eta_1(s)+\eta_2$, then with $\gamma_1=\eta_1(0)$ and $\gamma_2=\eta_2$ it is clear that the construction will simply give us back the maps $\eta_1(s)$ and $\eta_2(s)=\eta_2$.

It is mainly for this reason that we prefer to work with $\WR(\M)$ rather than $\RR(\M)$.
\end{remark}
\begin{proof} First of all we note that according to Theorem \ref{rectcont} the set 
$$K = \bigcup_{s \in [0,b_{\eta}]} {\rm supp}(\eta(s))$$
is compact. So below all measures belongs to 
$$\K =\left\{\eta \in \M : {\rm supp}(\eta) \subset K\right\},$$
which we know is a compact subset of $\M$. Furthermore it is easy to get the general statement from the case $t=0$ and $\MM =\N$, which we assume below.

If $b_{\eta}=0$, then there is nothing to prove, so we therefore now assume that $b_{\eta}>0$. For each $n$ we may divide $[0,b_{\eta}]$ into dyadic pieces
$$0 < \frac{1}{2^n}b_{\eta} < \ldots < \frac{k}{2^n}b_{\eta} < \ldots <\frac{2^n}{2^n}b_{\eta}=b_{\eta}.$$
Let us introduce 
$$D_n= \{s \in [0,b_{\eta}]: \textrm{ there is a number } k \textrm{ such that } s=kb_{\eta}/2^n\},$$
and
$$D = \left\{ s \in [0,b_{\eta}] : \textrm{ there are numbers } k,n \textrm{ such that } s=\frac{k}{2^n} b_{\eta}\right\} = \bigcup_{n=1}^{\infty} D_n.$$
Clearly $D_n$ increases with $n$ and $D$ is countable and dense in $[0,b_{\eta}]$. 

For each $N$ we may now apply Proposition \ref{splitcor} to the measures $\eta(s)$, $s \in D_N$ and our $\gamma_i$ to get decompositions of the form
\begin{itemize}
\item $\gamma_i= \sum_{j=1}^{\infty} \eta_{i,j}^N(0)$,
\item $\eta(s)= \sum_{i,j =1}^{\infty} \eta_{i,j}^N(s)$ for all $s \in D_N$,
\item $\sum_{i,j=1}^{\infty} \left| \|\eta_{i,j}^N((k+1)b_{\eta}/2^N)\| - \|\eta_{i,j}^N(kb_{\eta}/2^N)\|\right| \leq h(b_{\eta}/2^N)$ \newline for all $k=0,1,\ldots,2^N$,
\item ${\rm diam}\left( \supp(\eta_{i,j}^N((k+1)b_{\eta}/2^N)) \cup \supp(\eta_{i,j}^N(kb_{\eta}/2^N))\right) \leq b_{\eta}/2^N.$
\end{itemize}
Now we define 
$$\gamma_i^N(s) = \sum_{j=1}^{\infty} \eta_{i,j}^N(s).$$
Then we have by definition 
$${\rm d}_{\M}(\gamma_i^N((k+1)b_{\eta}/2^N), \gamma_i^N(kb_{\eta}/2^N)) \leq b_{\eta}/2^N.$$
Note that if we iterate this we actually have for every $n \leq N$
$${\rm d}_{\M}(\gamma_i^N((k+1)b_{\eta}/2^n), \gamma_i^N(kb_{\eta}/2^n)) \leq b_{\eta}/2^n.$$
Now let $F=(F_1,F_2) :\N \rightarrow \N \times D$ be a bijection.
Then there is a subsequence $\gamma_{F_1(1)}^{n_1}(F_2(1)), \gamma_{F_1(1)}^{n_2}(F_2(1)),\ldots$ of $\gamma_{F_1(1)}^{n}(F_2(1))$
(defined for all $n$ larger than the smallest $n$ for which $F_2(1)$ belongs to $D_n$) which converges to some $\gamma_{F_1(1)}(F_2(1))$. From the subsequence $\gamma_{F_1(2)}^{n_1}(F_2(2)),\gamma_{F_1(2)}^{n_2}(F_2(2)),\ldots$ we can now pick out a convergent subsequence which converges to some $\gamma_{F_1(2)}(F_2(2))$. If we proceed this way we hence end up with a family of measures $\gamma_i(s)$ for each $i \in \N$ and each $s \in D$.

First of all we note that for every $N \in \N$ we still have
$${\rm d}_{\M}(\gamma_i((k+1)b_{\eta}/2^N), \gamma_i(kb_{\eta}/2^N)) \leq b_{\eta}/2^N.$$
This is so simply because by definition of  $\gamma_i((k+1)b_{\eta}/2^N)$ and $\gamma_i(kb_{\eta}/2^N)$ there will be a subsequence $m_1,m_2,\ldots$of $\N$ such that $\gamma_i^{m_j}(kb_{\eta}/2^N)$ converges to $\gamma_i(kb_{\eta}/2^N)$ and $\gamma_i^{m_j}((k+1)b_{\eta}/2^N)$ converges to $\gamma_i((k+1)b_{\eta}/2^N)$ as $j \rightarrow \infty$. 
Since the corresponding inequality holds for $\gamma_i^{m_j}$ the statement follows.

This however implies that the maps $\gamma_i: D \rightarrow \K$ are $1$-Lipschitz, and hence we may uniquely extend them to such maps defined on $[0,b_{\eta}]$. These are the required curves.

\end{proof}
\begin{corollary}\label{split3}
Suppose $\MM$ is at most countable, $\eta \in \WR(\M)$, $\mu_i \in \M$ for each $i \in \MM$ and that
$$\mu|_{\bigcup_{t \in [0,b_{\eta}]}{\rm supp} (\eta(t))} \leq \sum_{i \in \MM} \mu_i \leq \mu.$$
Assume furthermore that $0 \leq t_0 <t_1<\ldots <t_m \leq b_{\eta}$.
Then there are curves $\nu_j \in \WR(\M)$, $j \in \N$ such that
\begin{itemize}
\item for each $j \in \N$ and $k \in \{0,1,\ldots,m\}$ there is $i_k \in \MM$ such that
$$\nu_j(t_k) \leq \mu_{i_k},$$
\item $\eta(s)= \sum_{j=1}^{\infty} \nu_j(s) \textrm{ for every } s \in [0,b_{\eta}].$
\end{itemize}
\end{corollary}
\begin{proof} Let $\mu_i = \phi_i\mu$.
If $m=0$, then we may apply Proposition \ref{split2} to the measures $\gamma_i = \phi_i \eta(t_0),$  and the statement follows. Now we proceed by induction. Suppose the statement holds up to $m-1$. Then there are curves $\nu_j'  \in \WR(\M)$ such that 
\begin{itemize}
\item for each $j \in \N$ and $k \in \{0,1,\ldots,m-1\}$ there is $i_k \in \MM$ such that
$$\nu_j'(t_k) \leq \mu_{i_k},$$
\item $\eta(s)= \sum_{j=1}^{\infty} \nu_j'(s) \textrm{ for every } s \in [0,b_{\eta}].$
\end{itemize}

Now we apply Proposition \ref{split2} again but to each of the curves $\nu_j'$, $t=t_m$ and  $\gamma_{(i,j)} = \phi_i \nu_j'(t_m)$.
This gives us curves $\eta_{(i,j)} \in \WR(\M)$ such that
\begin{itemize}
\item for each $j \in \N$ 
$$\eta_{(i,j)}(t_m) = \gamma_{(i,j)} \leq \mu_i,$$
\item $\nu_j'(s)= \sum_{i \in \MM} \eta_{(i,j)}(s) \textrm{ for every } s \in [0,b_{\eta}].$
\end{itemize}
Hence 
$$\eta(s) = \sum_{j \in \N} \nu_j'(s) = \sum_{(i,j) \in \MM \times \N} \eta_{(i,j)}(s).$$
So if $F: \N \rightarrow \MM \times \N$ is a bijection, then with $\nu_j = \eta_{F(j)}$ we get that
$$\eta(s) = \sum_{j=1}^{\infty} \nu_j(s) \textrm{ for all } s \in [0,b_{\eta}],$$
and also for every $j \in \N$ there is $(i,k)$ such that $\nu_j(s) = \eta_{(i,k)} \leq \nu'_k(s)$ holds for all $s$, and hence for every $k \in \{0,1,\ldots,m-1\}$ we have that there is, by the assumptions on $\nu_j'$, $i_k \in \MM$ such that $\nu_j(t_k) \leq \nu_j'(t_k) \leq \mu_{i_k}$. This finishes the proof.  
\end{proof}
Generating curves through measure preserving families of maps is crucial for our applications later. 
\begin{theorem}\label{Hmap2}
Suppose $c: [0,T] \rightarrow (0,1]$, for every $t \in [0,T]$ the map $H_t: X \rightarrow X$ is a homeomorphism and that the map $(t,x) \mapsto H_t^{-1}(x)$ is jointly continuous in $t$ and $x$. Suppose also that for every $t \in [0,T]$ 
$$c(t)\mu(H_t^{-1}(K)) =\mu(K) \textrm{ for all compact subsets } K \textrm{ of } X,$$ 
and that for all $s,t \in [0,T]$ we have that 
$$d_X(H_s^{-1}(x),H_t^{-1}(x)) \leq |s-t| \textrm{ for all } x \in X.$$
If $\phi: X \rightarrow [0,1]$ is measurable and such that $({\rm supp}(\phi))_T$ is a compact subset of $X$, then 
$$\eta(t) = c(t)(\phi \circ H_t)\mu \in \WR(\M).$$ 
\end{theorem}
\begin{proof}  Let $\vae >0$. We mimic the construction from the proof of Theorem \ref{Hmap1}, and define the measures 
$$\eta_i(s) =c(s)((\phi \chi_{A_i}) \circ H_s)\mu,$$
where ${\rm supp}(\phi)$ is a disjoint union of the sets $A_1,A_2,\dots,A_k$ which are measurable with diameter at most $\delta \in (0,\vae)$ such that 
\begin{align*}
& d_{X}(H_s^{-1}(x),H_s^{-1}(y)) \leq \vae\\
&\textrm{for all } s \in [0,T] \textrm{ and } x,y \in ({\rm supp}(\phi))_T \textrm{ such that } d_X(x,y) \leq \delta.
\end{align*}
Then, as in the proof of Theorem \ref{Hmap1}, we see that $\|\eta_i(s)\|$ is constant, and that
$${\rm diam}({\rm supp}(\eta_i(s)) \cup {\rm supp}(\eta_i(t))) \leq |s-t|+2\vae.$$
Hence
$$d_{\M}(\eta(t),\eta(s)) \leq |s-t| + 2\vae.$$
Since $\vae >0$ is arbitrary we see that $\eta:[0,T] \rightarrow \M$ is $1$-Lipschitz, which proves the statement. 
\end{proof}
\begin{example} Our most important example will be when $X$ is an open subset of $\R^n$, $d_X$ is the usual Euclidean norm and $\mu$ denotes the Lebesgue measure. 

The most important type of curve for us will be given by translation. Suppose $\overline{e}$ is a unit vector in $\R^n$ and $0 \leq \phi \leq 1$ where $\phi$ is a Borel measurable function with compact support in $X$. If we put
$$\eta(t) = \phi(\cdot + t\overline{e})\mu,$$
Then it follows from Theorem \ref{Hmap2} above that $\eta(t) \in \WR(\M)$ defined for $t \in [0,b_{\eta}]$ such that $\bigcup_{t \in [0,b_{\eta}]} {\rm supp}(\phi(\cdot+t\overline{e}))$ is a compact subset of  $X$, and in particular $d_{\M}(\eta(s),\eta(t))=|s-t|$.

It is also worthwhile to consider for a fixed $r>0$
$$\eta(t)=\left(\frac{r}{r+t}\right)^n \mu|_{B(x,r+t)}$$
(note that $\|\eta(t)\|$ is constant). Then it is again easy to see, using Theorem \ref{Hmap2}, that $d_{\M}(\eta(s),\eta(t))=
|s-t|$, and hence it forms a curve in $\WR(\M)$.

It should also be remarked that if we replace the Lebesgue measure by some other measure $\mu' =\psi \mu$ (where $\mu$ still denotes Lebesgue measure), in case there is a constant $c>0$ such that $\psi \geq c$, then $c\eta(t),$ where $\eta(t)$ is as above, belongs to $\M'$ where
$$\M'=\{\nu: 0 \leq \nu \leq \mu', \quad {\rm supp}(\nu) \textrm{ is compact.}\},$$
and it is easily seen to be a rectifiable curve in $\X'$, with the metric $d_{\M'}$. 
\end{example}

\section{$\L^p$-spaces on $\X$} \label{LP}
Let 
$$\X = \M \setminus \{0\},$$
and we give this set the induced metric $d_{\M}$.

For any function $F : \X  \rightarrow \overline{\R}$ we define the $\L^p(\X)$-norm:
\begin{equation}\label{Lpnorm}
\|F\|_{\L^p(\X)} = \sup\left\{\left(\sum_{i=1}^k |F(\eta_i)|^p \|\eta_i\|\right)^{1/p}:  \eta_i \in \X, \,  \supp(\eta_i) \cap \supp(\eta_j) = \emptyset \textrm{ if } i \ne j\right\}.
\end{equation}
We also introduce the space $\L^p(\X)$ to consist of all $F: \X  \rightarrow \R$ such that $\|F\|_{\L^p(\X)} < \infty$. Note that in case $|F(\eta)|=\infty$ for some $\eta$, then $||F||_{\L^p(\X)} = \infty$ trivially, so every function $F \in \L^p(\X)$ maps $\X$ into $\R$.
\begin{remark}
Although obvious it is worthwhile to note that there are no measureability assumptions on the functions $F$. Any function defined for all elements in $\X$ would do. In particular we do not need to worry about such issues when we do constructions like the upper semicontinuous regularization along curves $\breve{F}$ for instance. 
\end{remark}
\begin{lemma} $\L^p(\X)$ is a vector space, and $\|\cdot\|_{\L^p(\X)}$ is a norm on this space. \end{lemma}
\begin{proof}
Suppose $F,G \in \L^p(\X)$ and $a \in \R$.
If $\|F\|_{\L^p(\X)} =0$ then $|F(\eta)|^p\|\eta\|=0$ for every $\eta \in \X$. Hence $F(\eta)=0.$ 
It is also immediate by construction that $\|aF\|_{\L^p(\X)}=|a| \cdot\|F\|_{\L^p(\X)}$. Finally to prove the triangle inequality we have 
\begin{align*}
&\left(\sum_{i=1}^k |F(\eta_i)+G(\eta_i)|^p\|\eta_i\|\right)^{1/p} = \left(\sum_{i=1}^k |F(\eta_i)\|\eta_i\|^{1/p}+G(\eta_i)\|\eta_i\|^{1/p}|^p\right)^{1/p} \\
&\leq\left(\sum_{i=1}^k |F(\eta_i)\|\eta_i\|^{1/p}|^p\right)^{1/p} + \left(\sum_{i=1}^k |G(\eta_i)\|\eta_i\|^{1/p}|^p\right)^{1/p},
\end{align*}
where we in the last step simply applied Minkowski's inequality for the counting measure.

Hence we see that $F+G$ and $aF$ also belongs to $\L^p(\X)$, and the proof is done.
\end{proof}
\begin{lemma}
Suppose $F_j : \X \rightarrow \R$, $1 \leq j < \infty$ and $F,G : \X \rightarrow \R$. Then the following holds
\begin{itemize}
\item[(1)]  If $0\leq F_1 \leq F_2$ then $\|F_1\|_{\L^p(\X)} \leq \|F_2\|_{\L^p(\X)},$
\item[(2)] If  $0 \leq F_j \nearrow F \textrm{ then } \|F_j\|_{\L^p(\X)} \nearrow \|F\|_{\L^p(\X)},$
\item[(3)] $\|\sum_{j=1}^{\infty} F_j\|_{\L^p(\X)} \leq \sum_{j=1}^{\infty}\|F_j\|_{\L^p(\X)},$
\item[(4)] $\bigl\| |F|-|G| \bigr\|_{\L^p(\X)} \leq \|F-G\|_{\L^p(\X)},$
\item[(5)] If $\|F-F_j\|_{\L^p(\X)} \rightarrow 0$ as $j \rightarrow \infty$ then $\bigl\| |F|-|F_j| \bigr\|_{\L^p(\X)} \rightarrow 0$ as $j \rightarrow \infty$.
\end{itemize}
\end{lemma}
\begin{proof}
Statement $(1)$ is obvious. To prove $(2)$ we need to show that for any $\eta_i \in \X$, $1 \leq i \leq k$ such that $\supp(\eta_i) \cap \supp(\eta_j) = \emptyset \textrm{ if } i \ne j$ we have
$$\left(\sum_{i=1}^k |F(\eta_i)|^p\|\eta_i\| \right)^{1/p} \leq \lim_{j \rightarrow \infty} \|F_j\|_{\L^p(\X)}.$$
Given $\vae >0$ we may choose $J$ such that for all $j \geq J$ and each $i \in \{1,2,\ldots,k\}$ we have 
$F_j(\eta_i) \geq F(\eta_i) - \vae/(k\|\eta_i\|)^{1/p}.$
Hence we get
$$\left(\sum_{i=1}^k |F(\eta_i)|^p\|\eta_i\| \right)^{1/p} \leq \left(\sum_{i=1}^k \left|F_j(\eta_i)+\frac{\vae}{(k\|\eta_i\|)^{1/p}}\right|^p\|\eta_i\| \right)^{1/p} \leq \|F_j\|_{\L^p(\X)}+\vae,$$
and from this it is easy to see that the statement follows.

$(3)$ follows from $(1)-(2)$ together with the (finite) triangle inequality since  
\begin{align*}
&\|\sum_{j=1}^{\infty} F_j\|_{\L^p(\X)} 
\leq  \left\|\sum_{j=1}^{\infty} |F_j|\right\|_{\L^p(\X)} = \lim_{J \rightarrow \infty} \left\|\sum_{j=1}^{J} |F_j|\right\|_{\L^p(\X)} \\
&\leq \lim_{J \rightarrow \infty} \sum_{j=1}^{J}\|F_j\|_{\L^p(\X)} = \sum_{j=1}^{\infty}\|F_j\|_{\L^p(\X)}.
\end{align*}
$(4)$ follows from the fact that $\left||F(\eta)|-|G(\eta)|\right| \leq \left|F(\eta)-G(\eta)\right|$, because we then get
$$\sum_{i=1}^k\left||F(\eta_i)|-|G(\eta_i)|\right|^p\|\eta_i\| \leq \sum_{i=1}^k|F(\eta_i)-G(\eta_i)|^p\|\eta_i\|
\leq \|F-G\|^p_{\L^p(\X)}.$$
$(5)$ is an immediate consequence of $(4)$.
\end{proof}
\begin{theorem}[H\"o{}lder's inequality]
Suppose $1 \leq p,q \leq \infty$ and $\frac{1}{p}+\frac{1}{q} =1$.
Given $F,G : \X \rightarrow \R$, then $\|FG\|_{\L^1(\X)} \leq \|F\|_{\L^p(\X)}\|G\|_{\L^q(\X)}.$ 
\end{theorem}
\begin{proof}
Suppose $\supp(\eta_i) \cap \supp(\eta_j) = \emptyset \textrm{ if } i \ne j$, where each $\eta_i \in \X$. Then 
\begin{align*}
\sum_{i=1}^{k} |F(\eta_i)G(\eta_i)|\cdot\|\eta_i\| &= \sum_{i=1}^k |F(\eta_i)\|\eta_i\|^{1/p} |G(\eta_i)| \cdot\|\eta_i\|^{1/q} \\
&\leq \left(\sum_{i=1}^k |F(\eta_i)|^p\|\eta_i\|\right)^{1/p} \left(\sum_{i=1}^k |G(\eta_i)|^q \|\eta_i\|\right)^{1/q} \\
&\leq \|F\|_{\L^p(\X)}\|G\|_{\L^q(\X)},
\end{align*}
where we, to get the first inequality, applied H\"o{}lder's inequality for the counting measure.
\end{proof}

\begin{theorem}$\L^p(\X)$ is a Banach space.\end{theorem}
\begin{proof}
Suppose $F_j$ is a Cauchy sequence in $\L^p(\X)$, and assume without loss of generality that $\|F_{j+1}-F_j\|_{\L^p(\X)} < 2^{-j}$.
For any $\eta \in \X$ we have
\begin{align*}
&\|F_j -F_l\|^p_{\L^p(\X)} \\
&=\sup\left\{\sum_{i=1}^k |F_j(\eta_i)-F_l(\eta_i)|^p\|\eta_i\|: \eta_i \in \X,\, \supp(\eta_i) \cap \supp(\eta_j) = \emptyset \textrm{ if } i \ne j \right\}\\
&\geq |F_j(\eta)-F_l(\eta)|^p \|\eta\|.
\end{align*}
Hence we see that $F_j(\eta)$ forms a Cauchy sequence in $\R$ for each $\eta \in \X$, and hence $F_j(\eta) \rightarrow F(\eta)$ pointwise on $\X$ for some $F$. Then
\begin{align*}
\|F-F_j\|_{\L^p(\X)} &= \left\|\sum_{l=j}^{\infty}(F_{l+1}-F_l)\right\|_{\L^p(\X)} \\
&\leq \sum_{l=j}^{\infty}\|F_{l+1}-F_l\|_{\L^p(\X)} \leq \sum_{l=j}^{\infty} 2^{-l} = 2^{1-j},
\end{align*}
and hence we see that $F_j \rightarrow F$ in $\L^p(\X)$
\end{proof}

We will mainly be interested in those $F$ which in a natural sense corresponds to functions $f$ on $X$. 
To do so we first of all introduce for $f \in L^1_{\rm loc}(X)$ the function $F_f : \X \rightarrow \R$ by
$$F_f(\eta) = \frac{1}{\|\eta\|}\int f \,d \eta.$$
It is easy to see that $f \mapsto F_f$ is a linear operation, and also that if $F_f(\eta)=F_g(\eta)$ for all $\eta$ then $f=g$ $\mu$-a.e. Also note that
the map 
$$G_{F_f}(\eta) = \|\eta\|F_f(\eta)= \int f \, d\eta$$ 
has a natural extension to $\M$ if we define $G_{F_f}(0)=0$.
In the opposite direction we have the following:
\begin{lemma}\label{FFf}
Suppose  $F: \X \rightarrow \R$ and define $G_F : \M \rightarrow \R$ by 
$$G_F(\eta)=\left\{\begin{array}{ll}
\|\eta\|F(\eta) & \eta \in \X\\
0 & \eta=0,\end{array}\right.$$
Then $F$ is of the form $F_f$ for some $f \in L^{1}_{\rm loc} (X)$ if and only if $G_F$ satisfies
\begin{align}\label{Fadd1}
&G_F(t\eta) = tG_F(\eta) \textrm{ for all } t \in [0,1] \textrm{ and } \eta \in \M,\\ \label{Fadd2}
&G_F\left(\sum_{i=1}^{\infty}  \eta_i \right) = \sum_{i=1}^{\infty}G_F(\eta_i) \textrm{ for all } \eta_i \in \M
\textrm{ such that }\sum_{i=1}^{\infty}  \eta_i \in \M.
\end{align}
\end{lemma}
\begin{proof} That any $F$ of the form $F_f$ satisfies (\ref{Fadd1}) and (\ref{Fadd2}) is obvious.

To prove the opposite direction we note that for any fixed compact set $K \subset X$ the map
$$\gamma_K(A) = G_F(\mu|_A),$$
defined for all Borel sets $A \subset K$ by assumption satisfies
\begin{align*}
& \gamma_K(A)= 0 \textrm{ for all } A \subset X \textrm{ such that } \mu(A)=0,\\
& \gamma_K(\bigcup_{i=1}^{\infty} A_i) = \sum_{i=1}^{\infty} \gamma_K(A_i) \textrm{ for all disjoint families } A_i \subset K. 
\end{align*}
Hence $\gamma_K = f_K \mu|_K$ for some $f_K \in L^1(K)$. But if $K_1$ and $K_2$ are two  different compact sets, then since 
$\gamma_{K_1}(A) = \gamma_{K_2}(A)$ for all Borel sets $A \subset K_1 \cap K_2$ we see that $f_{K_1} = f_{K_2}$ on this intersection a.e. From this we may easily conclude that there is some $f \in L^1_{\rm loc}(X)$ such that $\gamma_K = f \mu|_K$ for any compact subset $K $ of $X$.

But this means in particular that 
$$G_F(\mu|_K) = \int_K f \,d\mu,$$
for any compact set $K$ in $X$. 
Also note that due to (\ref{Fadd2}) $G_F$ is order-continuous in the sense that if $\eta_n$ increases to $\eta$ then $G_F(\eta_n) \rightarrow G_F(\eta)$ as $n \rightarrow \infty$.
Together with (\ref{Fadd1}) and (\ref{Fadd2}) this is easily seen to imply that $F = F_f$.
\end{proof}

\begin{theorem} \label{eqlpnorm}
For any function $f \in L^1_{\rm loc}(X)$ we have that $f \in L^p(X)$ if and only if $F_f \in L^p(\X)$, and in that case $\|F_f\|_{\L^p(\X)} = \|f\|_{L^p(X)}$. 
\end{theorem}
\begin{proof}
We need to prove that 
$$\|f\|_{L^p(X)} = \sup\left\{ \left(\sum_{i=1}^k |F_f(\eta_i)|^p\|\eta_i\|\right)^{1/p}: \eta_i \in \X,\, \supp(\eta_i) \cap \supp(\eta_j) = \emptyset \textrm{ if } i \ne j\right\}.$$
From Jensen's inequality we get
\begin{align*}
&\left(\sum_{i=1}^k |F_f(\eta_i)|^p \|\eta_i\|\right)^{1/p} = \left(\sum_{i=1}^k \left| \frac{1}{\|\eta_i\|}\int f \, d \eta_i \right|^p \|\eta_i\|\right)^{1/p} \\
&\leq \left(\sum_{i=1}^k  \frac{1}{\|\eta_i\|}\int |f|^p \, d \eta_i \|\eta_i\|\right)^{1/p} \leq \left(\int |f|^p \, d \mu\right)^{1/p}.
\end{align*}
Hence $\|F_f\|_{\L^p(\X)} \leq \|f\|_{L^p(X)}$.

Suppose now that $\|f-f_n\|_{L^p(X)} \leq \vae$. Then we get
\begin{align*}
&\|F_f\|_{\L^p(\X)} = \|F_{f_n} - F_{f_n-f}\|_{\L^p(\X)} \geq \|F_{f_n}\|_{\L^p(\X)}-\|F_{f_n-f}\|_{\L^p(\X)} \\
&\geq \|F_{f_n}\|_{\L^p(\X)} - \|f_n-f\|_{L^p(X)}.
\end{align*}
In case we have $\|F_{f_n}\|_{\L^p(\X)} = \|f_n\|_{L^p(X)}$, then it would follow from the above that 
$$\|F_f\|_{\L^p(\X)} \geq \|f_n\|_{L^p(X)} -\vae \geq \|f\|_{L^p(X)}- 2\vae.$$
It is therefore enough to prove the statement for a dense subset of $L^p(X)$. It is however easy to verify the statement in case $f$ is continuous with compact support, and hence we get the result.
\end{proof}
We may now introduce
$$L^p(\X) = \{F \in \L^p(\X) : F \textrm{ satisfies } (\ref{Fadd1}) \textrm{ and } (\ref{Fadd2})\} =\{F_f: f \in L^p(X)\}.$$
Since $F_n=F_{f_n}$ is Cauchy in $\L^p(\X)$ if and only if $f_n$ is Cauchy in $L^p(X)$ it is clear that $L^p(\X)$ forms a closed subspace of $\L^p(\X)$, and in particular forms a Banach space itself with the same norm.
\begin{theorem} \label{contcurves}
If $f \in L^1_{\rm loc}(X)$ then $F_f$ is continuous along curves in $\WR(\X)$.
\end{theorem}
\begin{proof} This is a direct consequence of the fact that for any  $\eta \in \WR(\X)$ the set $\cup_{t \in [0,b_{\eta}]} \supp(\eta(t))$ is compact and the densities $\phi(s,\cdot)$, where $\eta(s)=\phi(s,\cdot)\mu$, converges weak$^*$ to $\phi(0,\cdot)$ in $L^{\infty}(X)$ as $s \rightarrow 0$. 
\end{proof}
\begin{theorem}
Suppose $F \in \L^p(\X)$ is non-negative. Then $\|F\|_{\L^{p}(\X)} = \|\breve{F}\|_{\L^p(\X)}$.
\end{theorem}
\begin{proof}
Obviously $\|F\|_{\L^{p}(\X)} \leq \|\breve{F}\|_{\L^p(\X)}$.
To prove the opposite inequality it is enough to show that 
$$\left(\sum_{i=1}^k |\breve{F}(\eta_i)|^p\|\eta_i\| \right)^{1/p} \leq \|F\|_{\L^{p}(\X)}$$
for all $\eta_i \in \X$ such that $\supp(\eta_i) \cap \supp(\eta_j)=\emptyset$ if $i \ne j$. 
Since the $\eta_i$ have disjoint compact supports there is $\vae>0$ which is smaller than the distance between all these as elements in $\M$. For any such $\vae$
we may hence choose curves $\eta_i(s)$ such that $\eta_i(0)=\eta_i$ and $\breve{F}(\eta_i) \leq  F(\eta_i(s_i))+\vae/(k\|\eta_i\|)^{1/p}$
for some $s_i \in [0,\vae/3]$ for instance.
 
By construction, and an application of Theorem \ref{rectcont}, we see that the measures $\eta_i'=\eta_i(s_i)$ have disjoint supports and we get
\begin{align*}
&\left(\sum_{i=1}^k |\breve{F}(\eta_i)|^p\|\eta_i\|\right)^{1/p} \leq \left(\sum_{i=1}^k \left|F(\eta_i')+\frac{\vae}{(k\|\eta_i\|)^{1/p}} \right|^p\|\eta_i\|\right)^{1/p}\\
&\leq \left(\sum_{i=1}^k |F(\eta_i')|^p \|\eta_i'\|\right)^{1/p} + \vae \leq \|F\|_{\L^p(\X)} + \vae.
\end{align*}
Since $\vae>0$ is arbitrary we get the result.
\end{proof}
\subsection{The space $\L^p_{\rm loc}(\X)$}\label{lploc}
It will be convenient to also have local spaces, and they are defined in essentially the obvious way. Suppose $K \subset X$ is compact, and let 
$$\K =\{\eta \in \X : {\rm supp}(\eta) \subset K\} \subset \X.$$
We may then regard $(K,d_X,\mu|_K)$ as our space, and we define $\L^p(\K)$ as above for each such $K$. We then say that $F : \X \rightarrow \R$ belongs to $\L^p_{\rm loc}(\X)$ if the restriction to $\K$ belongs to $\L^p(\K)$ for each compact subset $K$ of $X$, and similarly for the spaces $L^p_{\rm loc}(\X)$.

Note that the natural analogue of Theorem \ref{eqlpnorm} still holds in this situation in the following sense:
\begin{theorem} \label{eqlpnorm2}
For any function $f \in L^1_{\rm loc}(X)$ we have that $f \in L^p_{\rm loc}(X)$ if and only if $F_f \in L^p_{\rm loc}(\X)$, and in that case $\|F_f\|_{\L^p(\K)} = \|f\|_{L^p(K)}$ for each compact subset $K$ of $X$. 
\end{theorem}
\section{Upper gradients}\label{upgrad}
Let $F : \X \rightarrow \R$ and $\eta \in \M$.
We introduce the $\WR(\X)$-upper gradients as follows. If $\vae>0$ then we put
$$r_F^{\vae}(\eta) := \sup\left\{\frac{|F(\nu(s)) -F(\eta)|}{s} : 
 \nu \in \WR(\X) \textrm{ such that } \nu(0)=\eta \textrm{ and } 0< s < \vae \wedge b_{\nu} \right\},$$
and then we define
$$r_F(\eta)= r_F^{0}(\eta) = \lim_{\vae \rightarrow 0} r_F^{\vae}(\eta).$$
Since $r^{\vae}_F$ decreases as $\vae$ decreases this is well defined. Also note that 
$\nu(s)=\eta$ for all $s$ is an element of $\WR(\X)$, so the definition always makes sense, even if there are no non-constant rectifiable curves starting at $\eta$. In this case it is furthermore clear that we would have $r_F(\eta)=0$. 

Note in particular that for any $\nu \in \WR(\X)$ and $0<|s-t| < \vae$ we have
$$\left| \frac{F(\nu(s))-F(\nu(t))}{s-t}\right| \leq r_F^{\vae}(\nu(t)),$$
hence
$$\limsup_{s \rightarrow t} \left| \frac{F(\nu(s))-F(\nu(t))}{s-t}\right| \leq r_F^{\vae}(\nu(t)),$$
and since this holds for any $\vae >0$ it also holds for $\vae=0$.

It is also clear that 
$$\lim_{\delta \rightarrow \vae^-} r^{\delta}_F(\eta) =r^{\vae}_F(\eta).$$
\begin{remark}\label{contcrem}
Suppose for a given $\eta \in \X$ that $r_F(\eta) < \infty.$ Then by definition there is some $\vae >0$ such that $r_F^{\vae}(\eta) < \infty$.
In case $\nu \in \WR(\M)$ is such that $\nu(0) = \eta$, then it follows more or less immediately from the definition, since $|F(\nu(s)) - F(\eta)| \leq r_F^{\vae}(\eta) s$ for all $s < \vae \wedge b_{\nu}$,  that $\lim_{s \rightarrow 0} F(\nu(s)) = F(\eta)$.
\end{remark}
\begin{remark}
It is easy to see that the definition would not change if we replaced $\WR(\X)$ by $\RR(\X)$ since if we reparametrize a curve in $\WR(\X)$ so it becomes parametrized by arc-length, then the corresponding map $r : [0,l_{\nu}] \rightarrow [0,b_{\nu}]$ satisfies $r(t) \geq t$ by definition.

The only reason we choose to work with $\WR(\X)$ instead is that it behaves better when we decompose curves such as in Proposition \ref{split2}.
\end{remark}

\begin{lemma}\label{maxalc}
If $F: \X \rightarrow \R$ and $\eta \in \X$, then there is a curve $\nu$ in $\WR(\X)$ such that $\nu(0)=\eta$ and
$$r_F(\eta) = \limsup_{s \rightarrow 0} \left|\frac{F(\nu(s))-F(\nu(0))}{s}\right|$$
\end{lemma}
\begin{proof}
This proof is more or less analogous to that of Lemma \ref{maxalc1}, but we give the details for completeness.
If $r_F(\eta)=0$ then the result is trivial.
Assume now that $0 < r_F(\eta) < \infty$. Then we may by definition inductively choose sequences $\delta_n,\vae_n,$ decreasing to zero and curves $\nu_n \in \WR(\X)$ such that
\begin{itemize}
\item[(1)] $0< \delta_n \leq \vae_n$,
\item[(2)] $\nu_n(0)= \eta$,
\item[(3)] $r_F^{\vae_n}(\eta) \leq \left|\frac{F(\nu_n(\delta_n))-F(\nu(0))}{\delta_n}\right|+\frac{1}{n}$,
\item[(4)] $\vae_{n+1} \leq 2^{-n} \delta_n$.
\end{itemize}
Also put
$$k_n= \delta_n + 2\sum_{j=1}^{n-1} \delta_j, \quad b=2\sum_{j=1}^{\infty} \delta_j.$$
We define $\gamma: [0,b] \rightarrow \X$ such that 
$$\gamma(t)= \left\{ 
\begin{array}{ll} 
\nu_n(t-k_n+\delta_n), & t \in [k_n-\delta_n,k_n)\\
\nu_n(k_n+\delta_n - t), & t \in [k_n,k_n+\delta_n)\\
\eta, & t=b.\end{array}\right.$$
Finally put $\nu(t)=\gamma(b-t)$. It is easy to verify that $\nu \in \WR(\X)$ ($\gamma$ simply consists of rectifiable curves subparametrized by arclength going back to forth from $\eta$ and then patched together). If we put 
$$r_n=\delta_n + 2 \sum_{j=n+1}^{\infty} \delta_j,$$
then 
$$\nu(r_n)=\gamma(b-r_n) =\gamma\left(\delta_n+2\sum_{j=1}^{n-1}\delta_j\right) = \gamma(k_n) =\nu_n(\delta_n).$$
Also note that
$$\lim_{n \rightarrow \infty} \frac{\delta_n}{r_n}=1.$$
Hence we get
\begin{align*}
&\limsup_{s \rightarrow 0} \left|\frac{F(\nu(s))-F(\nu(0))}{s}\right|  \\
&\geq\limsup_{n \rightarrow \infty} \left|\frac{F(\nu(r_n))-F(\nu(0))}{r_n}\right|=
\limsup_{n \rightarrow \infty} \left|\frac{F(\nu(r_n))-F(\nu(0))}{r_n}\right|+ \limsup_{n \rightarrow 0} \frac{1}{n} \\
& = \limsup_{n \rightarrow \infty} \left(\left|\frac{F(\nu_n(\delta_n))-F(\nu_n(0))}{\delta_n} \right| + \frac{r_n}{n\delta_n}\right)\frac{\delta_n}{r_n} \geq \limsup_{n \rightarrow 0} r_F^{\vae_n}(\eta)\frac{\delta_n}{r_n} =r_F(\eta).
\end{align*}
The case $r_F(\eta)=\infty$ is treated similarly but replacing (3) above by 
$$\left|\frac{F(\nu_n(\delta_n))-F(\nu(0))}{\delta_n}\right| \geq n.$$
\end{proof}
\begin{theorem}\label{fundth2}
If $\nu \in \WR(\X)$, $s \in [0,b_{\nu}]$ and $F: \X \rightarrow \R$ then 
$$|F(\nu(s))-F(\nu(0))| \leq \int_0^s \breve{r}_F(\nu(t))\,dt.$$
Furthermore, if $\eta \in \X$ and $g : \X \rightarrow [0,\infty)$ satisfies

\begin{align*}
&\limsup_{s \rightarrow 0}g(\nu(s)) \leq g(\eta),\\
&|F(\nu(s))-F(\nu(0))| \leq \int_0^s g(\nu(t))\, dt
\end{align*}
for every curve $\nu \in \WR(\X)$ with $\nu(0)=\eta$, then
$$\breve{r}_F(\eta) \leq g(\eta).$$
\end{theorem}
\begin{proof}
The first part is more or less a direct consequence of Lemma \ref{fundth}, if we define $f(s)=F(\nu(s))$ and put $a= b_{\nu}$.
By assumption we have
$$\limsup_{s \rightarrow t} \left| \frac{f(s)-f(t)}{s-t}\right| = 
\limsup_{s \rightarrow t} \left| \frac{F(\nu(s))-F(\nu(t))}{s-t}\right| \leq r_F(\nu(t)),$$
and since $r_F(\nu(t)) \leq \breve{r}_F(\nu(t))$ the first part is proved.

For the second part we apply Lemma \ref{maxalc} to get that there is a curve $\nu \in \WR(\X)$ such that $\nu(0)=\eta$ and
$$r_F(\eta) \leq \limsup_{s \rightarrow 0} \left|\frac{F(\nu(s))-F(\nu(0))}{s}\right|.$$
But since
$$|F(\nu(s))-F(\nu(0))| \leq \int_0^s g(\nu(t))\, dt$$
we also have 
$$\limsup_{s \rightarrow 0} \left|\frac{F(\nu(s))-F(\nu(0))}{s}\right| \leq \limsup_{s \rightarrow 0} \frac{1}{s} \int_0^s g(\nu(t))\,dt \leq g(\nu(0)) =g(\eta).$$
Hence 
$$r_F(\eta) \leq g(\eta),$$
and, by an application of Lemma \ref{maxalc1}, the proof is done.
\end{proof}
\begin{remark}
Note that, according to the proof above, in case $g : \X \rightarrow [0,\infty)$ is upper semicontinuous along curves, and
$$\limsup_{s \rightarrow 0} \left| \frac{F(\eta(s))-F(\eta(0))}{s}\right| \leq g(\eta(0))$$
holds for every $\eta \in \WR(\X)$, then $\breve{r}_F \leq g$.
\end{remark}
\begin{theorem}\label{ugbasic}
Suppose $F,G : \X \rightarrow \R$.  
\begin{itemize}
\item[(1)] If $\vae \geq 0$ then $r_{aF}^{\vae} = |a|r_F^{\vae}  \textrm{ for all } a \in  \R,$
\item[(2)] If $\vae \geq 0$ then $r_{F+G}^{\vae} \leq r_F^{\vae} + r_G^{\vae},$
\item[(3)] If $\vae >0$ then $r_{FG}^{\vae} \leq |F|r_G^{\vae} + |G|r_F^{\vae} + r_F^{\vae}r_G^{\vae}\vae,$
\item[(4)] $r_{FG} \leq |F|r_G + |G|r_F$,
\item[(5)] If $\eta \in \X$ and $r_F(\eta) < \infty$ then ${r}_{|F|}(\eta)= {r}_F(\eta)$.
\end{itemize}
\end{theorem}

\begin{proof} Let $\eta \in \WR(\X)$.\newline
(1) If $a=0$ this is self-evident. Otherwise it follows for $\vae>0$ since 
\begin{align*}
&|aF(\eta(s))-aF(\eta(0))|\leq |a|ks \Leftrightarrow |a\|F(\eta(s))-F(\eta(0))| \leq |a|ks\\
& \Leftrightarrow |F(\eta(s))-F(\eta(0))| \leq ks.
\end{align*}
(2) follows for $\vae>0$ since if $|F(\eta(s))-F(\eta(0))| \leq k_1s$ and
$|G(\eta(s))-G(\eta(0))| \leq k_2s$ then
\begin{align*}
&|F(\eta(s))+G(\eta(s))-F(\eta(0))-G(\eta(0))| \\
&\leq |F(\eta(s))-F(\eta(0))|+|G(\eta(s))-G(\eta(0))| \leq (k_1+k_2)s.
\end{align*}
That (1) and (2) also holds for the value $\vae =0$ follows directly by just taking limits. 
(3) Suppose $s < \vae \wedge b_{\eta}$. Then
\begin{align*}
&|F(\eta(0))G(\eta(0))-F(\eta(s))G(\eta(s))| \\ &= |F(\eta(0))(G(\eta(0))-G(\eta(s))) + (G(\eta(s))-G(\eta(0)))(F(\eta(0))-F(\eta(s)))\\
&+ G(\eta(0))(F(\eta(0))-F(\eta(s)))| \\
& \leq|F(\eta(0))\|G(\eta(0))-G(\eta(s))|+|G(\eta(0))\|F(\eta(0))-F(\eta(s))|\\
& + |G(\eta(s))-G(\eta(0))\|F(\eta(s))-F(\eta(0))| \\
& \leq \left(|F(\eta(0))|r_G^{\vae}(\eta(0)) + |G(\eta(0))|r_F^{\vae}(\eta(0)) + r_F^{\vae}(\eta(0))r_G^{\vae}(\eta(0))s\right)s.
\end{align*}
(4) follows from (3) by taking the limit $\vae \rightarrow 0$.\smallskip\newline
To prove (5) we apply Lemma \ref{maxalc} to first get that there is a curve $\nu \in \WR(\X)$ such that 
$\nu(0)=\eta$ and 
$$r_F(\eta)=\limsup_{s \rightarrow 0} \left| \frac{F(\nu(s))-F(\eta)}{s}\right|.$$
Since the assumption that $r_F(\eta)<\infty$ implies that $F(\nu(s))$ is continuous at $s=0$, for $s$ close to $0$ we always have $\left| |F(\nu(s))|-|F(\eta)|\right|=\left|F(\nu(s))-F(\eta)\right|$. Therefore we see that 
$$r_F(\eta)=\limsup_{s \rightarrow 0} \left| \frac{F(\nu(s))-F(\eta)}{s}\right| = \limsup_{s \rightarrow 0} \left| \frac{|F(\nu(s))|-|F(\eta)|}{s}\right| \leq r_{|F|}(\eta).$$
Reversing the roles of $F$ and $|F|$ above gives the opposite inequality.
\end{proof}
The corresponding result holds more or less immediately by definition also for the usc regularized gradients:
\begin{theorem}\label{ugregbasic}
Suppose $F,G : \X \rightarrow \R$.  
\begin{itemize}
\item[(1)] If $\vae \geq 0$ then $\breve{r}_{aF}^{\vae} = |a|\breve{r}_F^{\vae} \textrm{ for all } a \in  \R,$
\item[(2)] If $\vae \geq 0$ then $\breve{r}_{F+G}^{\vae} \leq \breve{r}_F^{\vae} + \breve{r}_G^{\vae},$
\item[(3)] If $\vae >0$ then $\breve{r}_{FG}^{\vae} \leq \breve{|F|}\breve{r}_G^{\vae} + \breve{|G|}\breve{r}_F^{\vae} + \breve{r}_F^{\vae}\breve{r}_G^{\vae}\vae,$
\item[(4)] $\breve{r}_{FG} \leq \breve{|F|}\breve{r}_G + \breve{|G|}\breve{r}_F$,
\item[(5)] If $\eta \in \X$ and $r_F(\eta) < \infty$ then $\breve{r}_{|F|}(\eta)=  \breve{r}_F(\eta)$.
\end{itemize}
\end{theorem}

An important property of these upper gradients, which in particular will be used to prove completeness of our Sobolev-type spaces below, is as follows:

\begin{theorem}\label{countadd}
If $F=\sum_{n=1}^{\infty} F_n$ and $\vae > 0$ then $r_F^{\vae} \leq \sum_{n=1}^{\infty} r_{F_n}^{\vae}$.

Furthermore $\breve{r}_F \leq \breve{g}$, where $g= \sum_{n=1}^{\infty} \breve{r}_{F_n}$.
\end{theorem}
\begin{proof}
For $\vae>0$, $\eta \in \WR(\X)$ and $s < \vae \wedge b_{\eta}$ we have
 \begin{align*}
&|F(\eta(s)) - F(\eta(0))| = \left| \sum_{n=1}^{\infty} F_n(\eta(s)) - \sum_{n=1}^{\infty} F_n (\eta(0))\right| \\
&\leq \sum_{n=1}^{\infty} |F_n(\eta(s)) - F_n(\eta(0))| \leq \left(\sum_{n=1}^{\infty} r_{F_n}^{\vae}(\eta(0)) \right)s.
\end{align*} 
As for the second part we note that 
\begin{align*}
&|F(\eta(s)) - F(\eta(0))| \leq \sum_{n=1}^{\infty} |F_n(\eta(s))-F_n(\eta(0))| \leq \sum_{n=1}^{\infty} \int_0^s \breve{r}_{F_n}(\eta(t)) \, dt\\
& = \int_0^s \left( \sum_{n=1}^{\infty} \breve{r}_{F_n}\right)(\eta(t)) \, dt = \int_0^s g(\eta(t)) \, dt \leq \int_0^s \breve{g}(\eta(t))\,  dt.
\end{align*}
\end{proof}
\begin{theorem}
Suppose $F : \X \rightarrow \R$, $\nu \in \X$, $r_F(\nu) < \infty$ and that $f$ is continuously differentiable in some neighborhood of $F(\nu)$. Then 
$$r_{f \circ F}(\nu)=f'(F(\nu))r_F(\nu) \textrm{ and } \breve{r}_{f \circ F}(\nu)=f'(F(\nu))\breve{r}_F(\nu).$$
\end{theorem}

\begin{proof}
We know that $|f(x)-f(y)| \leq \sup_{t \in [0,1]}|f'(tx+(1-t)y)|\cdot|x-y|$ for any points $x,y \in \R$. Hence for any curve  $\eta$ such that $\eta(0)=\nu$ we have
$$|f(F(\eta(s)) - f(F(\eta(0))| \leq \sup_{t \in [0,1]} |f'(tF(\eta(s)) + (1-t)F(\eta(0)))|\cdot|F(\eta(s))-F(\eta(0))|.$$
But as $s \rightarrow 0$ this implies that $r_{f \circ F}(\eta(0)) \leq |f'(F(\eta(0))|r_F(\eta(0))$.

If $f'(F(\eta(0))=0$ this  must be an equality. Otherwise $f$ is invertible in some neighborhood of $F(\eta(0))$, and if we apply the formula to $f^{-1}\circ f \circ F$ we get
$$r_{f^{-1} \circ f \circ F} (\eta(0)) \leq (f^{-1})'(f(F(\eta(0)))r_{f \circ F}(\eta(0)),$$
and since $(f^{-1})'(f(F(\eta(0))) = (f'(F(\eta(0)))^{-1}$ the opposite inequality also follows.

The statement about the upper semicontinuous regularizations follows immediately by definition. 
\end{proof}
\begin{lemma}\label{lattice}
If $F,G : \X \rightarrow \R$ then 
\begin{itemize}
\item[(1)] $\breve{r}_{F \vee G} \leq \breve{r}_F \vee \breve{r}_G \leq \breve{r}_F + \breve{r}_G,$
\item[(2)]$\breve{r}_{F \wedge G} \leq \breve{r}_F \vee \breve{r}_G \leq \breve{r}_F + \breve{r}_G,.$
\end{itemize}
\end{lemma}
\begin{remark}
Note that the second formula above has $\wedge$ in the left hand side, but $\vee$ on the right hand side. It is certainly not possible to replace $\vee$ with $\wedge$ here (for instance, if $F \leq G$ and $G$ is constant, then $\breve{r}_{F \wedge G} = \breve{r}_F$, but $\breve{r}_F \wedge \breve{r}_G=0$).
\end{remark}
\begin{proof}
Let $\eta \in \WR(\X)$. In case either of $r_F((\eta(0))$ or $r_G(\eta(0))$ is infinite the inequalities holds trivially, so we may assume that both of these are finite. According to Remark \ref{contcrem} we know that this implies that $F(\eta(s))$ and $G(\eta(s))$ as functions of $s$ are continuous at $0$.
Suppose $F(\eta(0)) > G(\eta(0))$, then $F(\eta(s)) > G(\eta(s))$ for $s$ close to $0$ as-well, and hence
\begin{align*}
&\limsup_{s \rightarrow 0} \left| \frac{F(\eta(s)) \vee G(\eta(s)) - F(\eta(0)) \vee G(\eta(0))}{s}\right| = \limsup_{s \rightarrow 0} \left| \frac{F(\eta(s)) - F(\eta(0))}{s}\right| \\
&\leq \breve{r}_F(\eta(0)) \leq \breve{r}_F(\eta(0)) \vee \breve{r}_G(\eta(0)).
\end{align*} 
A similar estimate holds also in case $G(\eta(0)) > F(\eta(0))$.

If on the other hand $F(\eta(0))=G(\eta(0))$ then it is easy to see that 
\begin{align*}
&\limsup_{s \rightarrow 0} \left| \frac{F(\eta(s)) \vee G(\eta(s)) - F(\eta(0)) \vee G(\eta(0))}{s}\right| \\
&\leq \limsup_{s \rightarrow 0} \left| \frac{F(\eta(s)) - F(\eta(0))}{s}\right| \vee \limsup_{s \rightarrow 0} \left| \frac{G(\eta(s)) - G(\eta(0))}{s}\right| \\
& \leq \breve{r}_F(\eta(0)) \vee \breve{r}_G(\eta(0)).
\end{align*}
The case of $F \wedge G$ is treated similarly.
\end{proof}
\subsection{Upper gradients of functions on $X$}
A problem with the upper gradients $r_F$ is that they are not functions on $X$ a-priori even if $F$ is of the form $F_f$ for some $f \in L^1_{\rm loc}(X)$. Our first objective is to prove that there is a natural function on $X$ which represents this gradient in case $F =F_f$ for some $f \in L^1_{\rm loc}(X)$.

Before we prove this, we start by proving the monotonicity of $r_{F_f}^{\vae}$ and $\breve{r}_{F_f}$.
\begin{lemma}\label{monlem}
Given $f \in L^1_{\rm loc}(X)$ and elements $\eta_1,\eta$ in $\X$ such that $\eta_1 \leq \eta$ then
\begin{itemize}
\item[(1)] $r^{\vae}_{F_f}(\eta_1) \leq r^{\vae}_{F_f}(\eta) \quad \textrm{ for all } \vae \in[0,\infty)$,
\item[(2)] $\breve{r}_{F_f}(\eta_1) \leq \breve{r}_{F_f}(\eta)$
\end{itemize} 
\end{lemma}
\begin{proof}
Suppose $\nu \in \WR(\X)$ with $\nu(0)=\eta$, and let $t \in (0,1]$, then $t\nu \in \WR(\X)$ and we have
$$|F_f(t\nu(s)) -F_f(t(\nu(0))| = |F_f(\nu(s))-F_f(\nu(0))|.$$
Hence it follows that $r^{\vae}_{F_f}(t\eta)=r^{\vae}_{F_f}(\eta)$ for all $t \in (0,1]$, and  $\vae >0$. Passing to the limit gives also $r_{F_f}(\eta)=r_{F_f}(t\eta)$. It also follows immediately by definition that we have $\breve{r}_{F_f}(t\eta) = \breve{r}_{F_f}(\eta).$

Suppose now that $\nu_1 \in \WR(\X)$ with $\nu_1(0)=\eta_1$, and let $\eta_2=\eta-\eta_1$. It is straightforward to see that the curve $\nu(s)= \frac{1}{2}(\nu_1(s)+\eta_2)$ belongs to $\WR(\X)$,  
$\nu(0)= \eta/2$ and that it satisfies  $d_{\M}(\nu(s),\nu(0)) \leq s$ for each $s$ (note that we can not expect any improvement on this, since the distance $d_{\M}$ typically is controlled by the relation between the supports of the measures rather than the total masses).
Hence we get if $\vae >0$ and $s < \vae \wedge b_{\nu_1}$:
\begin{align*}
& |F_f(\nu_1(s))-F_f(\nu_1(0))| =|(F_f(\nu_1(s)/2) - F_f(\nu_1(0)/2) + (F_f(\eta_2/2) -F_f(\eta_2/2))| \\
&=|F_f((\nu_1(s)+\eta_2)/2) - F_f((\nu_1(0)+\eta_2)/2)| \leq r^{\vae}_{F_f}(\eta/2) = r^{\vae}_{F_f}(\eta).
\end{align*}
This proves the statement for $\vae >0$. The rest of the statements follows directly by taking the limit as $\vae \rightarrow 0$ and the definition of the upper semicontinuous regularization. 
\end{proof}
\begin{theorem}
Suppose $f \in L^1_{\rm loc}(X)$. Then there is a unique element $g_f \in L^1_{\rm loc}(X)$ such that 
$$r_{F_f} =F_{g_f}$$
if and only if $r_{F_f} \in \L^{1}_{\rm loc}(\X)$.

In particular, if $r_{F_f} \in \L^1_{\rm loc}(\X)$, then $r_{F_f}$ is continuous along curves, and hence $r_{F_f}=\breve{r}_{F_f}$.
\end{theorem}
\begin{remark}
A consequence of this is that $r_{F_f}=\breve{r}_{F_f} \in \L^{p}(\X)$ if and only if $g_f \in L^p(X)$, and the norms are the same.

Also note that $g_f$ then satisfies for every $\eta \in \WR(\X)$ and $s \in [0,b_{\eta}]$
$$ \left|\int f \, d\eta(s)-\int f d\eta(0)\right| \leq \int_0^s \left(\int g_f\,  d\eta(t)\right)\, dt.$$
\end{remark}
\begin{proof}
That $r_{F_f} \in \L^{1}_{\rm loc}(\X)$ is a necessary condition is self-evident considering Theorem \ref{eqlpnorm2}. 
If this is satisfied however, then the map $r_{F_f}$ is finite valued and 
we need to prove that the function $G_{r_{F_f}}$ satisfies (\ref{Fadd1}) and (\ref{Fadd2}).

To prove that it satsifies (\ref{Fadd1}) we simply note that $G_{r_{F_f}}(t\eta) = \|t\eta\|r_{F_f}(t\eta)=t\|\eta\|r_{F_f}(\eta)=tG_{r_{F_f}}(\eta)$, where we used the result $r_{F_f}(t\eta)=r_{F_f}(\eta)$ for $ 0<t\leq 1$ as we saw in the previous proof. For $t=0$ there is nothing to prove.

To prove (\ref{Fadd2}) we do it in two steps. Since for each $i$ we have, by Lemma \ref{monlem}, that $r_{F_f}(\eta_i) \leq r_{F_f}\left(\sum_{i=1}^{\infty} \eta_i\right)$ we get  
\begin{align*}
&\sum_{i=1}^{\infty}G_{r_{F_f}}(\eta_i) = \sum_{i=1}^{\infty} \|\eta_i\|r_{F_f}(\eta_i) \\
&\leq \sum_{i=1}^{\infty} \|\eta_i\| r_{F_f}\left(\sum_{i=1}^{\infty} \eta_i\right)=
\|\sum_{i=1}^{\infty} \eta_i\| r_{F_f}\left(\sum_{i=1}^{\infty}\eta_i\right) = G_{r_{F_f}}\left(\sum_{i=1}^{\infty} \eta_i\right).
\end{align*}
To prove the opposite inequality we appeal to Proposition \ref{split2}. Using the notation from that proposition we get for any $\vae >0$ and $s \leq \vae \wedge b_{\eta}$
\begin{align*}
&\left| \|\eta(s)\| F_f(\eta(s)) -\|\eta(0)\| F_f(\eta(0))\right| =  \left|\sum_{i=1}^{\infty} (\|\eta_i(s)\| F_f(\eta_i(s)) - \|\eta_i(0)\|F_f(\eta_i(0))\right| \\
&\leq \sum_{i=1}^{\infty} \|\eta_i(0)\| \cdot |F_f(\eta_i(s)) - F_f(\eta_i(0))| \leq \sum_{i=1}^{\infty} \|\eta_i(0)\| r^{\vae}_{F_f}(\eta_i(0)) 
\end{align*}
(Note that for some $\vae >0$ the value $r_{F_f}^{\vae}(\eta(0))$ is finite, since otherwise $r_{F_f} (\eta(0))$ would also be infinite, and hence have infinite $\L^{p}(\X)$-norm.)
Therefore we have for all $\vae >0$ small enough that 
$$\|\eta(0)\|r^{\vae}_{F_f}(\eta(0)) \leq \sum_{i=1}^{\infty} \|\eta_i(0)\|r_{F_f}^{\vae}(\eta_i(0)).$$If we simply let $\vae \rightarrow 0$ on both sides we see that this also holds for $\vae =0$, and this is exactly the statement 
$$G_{r_{F_f}}(\eta(0)) \leq \sum_{i=1}^{\infty}G_{r_{F_f}}(\eta_i(0)).$$
\end{proof}
\begin{proposition}\label{fundprop}
Suppose $f \in L^1_{\rm loc}(X)$ and $g \in L^1_{\rm loc}(X)$ is non-negative.
If $g$ satisfies 
\begin{equation}\label{fundest}
 \left|\int f \, d\eta(s)-\int f \,  d\eta(0)\right| \leq \int_0^s \left(\int g \, d\eta(t)\right) \, dt
 \end{equation}
for every $\eta \in \WR(\X)$ and $s \in [0,b_{\eta}]$, then $r_{F_f} \in \L^{1}_{\rm loc}(\X)$ and $g_f \leq g$ $\mu$-a.e. 
\end{proposition}
\begin{remark}
Hence $g_f$ is $\mu$-a.e. the smallest such function $g$ that satisfies the above estimate.
Also note that if (\ref{fundest}) holds for all $\eta \in \WR(\X)$, then it also holds for all $\eta \in \WR(\M)$ since the only curve in $\WR(\M)$ that does not belong to $\WR(\X)$ is identically zero for which the statement trivially is true.
\end{remark}
\begin{proof}
This is more or less immediate from the definitions, since this implies that $\breve{r}_{F_f} \leq F_g$.
\end{proof}
\begin{proposition}\label{lipgrad}
Suppose $f$ is Lipschitz continuous on $X$ with Lipschitz constant $C$, then $g_f \leq C$ $\mu$-a.e.
\end{proposition}
\begin{proof}
If $\eta \in \WR(\X)$, and we decompose $\eta(t)= \sum_{i=1}^{\infty} \nu_i(t)$ in such a way that the diameter of ${\rm supp}(\nu_i(s)) \cup {\rm supp}(\nu_i(0))$ is at most $(1+\vae)s$ say where $\vae >0$, then
\begin{align*}
& \left| \int f \, d\eta(s)- \int f \, d\eta(0)\right| =  \left|\sum_{i=1}^{\infty} \left( \int f \, d\nu_i(s)- \int f \, d\nu_i(0)\right)\right| \\
&\leq  \left| \sum_{i=1}^{\infty} C(1+\vae)s||\nu_i(0)||\right| =C\|\eta(0)\|(1+\vae)s.
\end{align*}
Since $\vae>0$ is arbitrary the statement follows from Proposition \ref{fundprop}.
\end{proof}
\begin{proposition}\label{ugfb1}
Suppose $f,h \in L^1_{\rm loc}(X)$ and $a \in \R$. Then the following holds
\begin{itemize}
\item[(1)] $g_{f+h} \leq g_f + g_h$,
\item[(2)] $g_{af} =|a|g_{f}$,
\end{itemize}
\end{proposition}
\begin{proof}
This follows more or less immediately from the definitions and Theorem \ref{ugbasic}, since the map $f \mapsto F_f$ is linear.
\end{proof}
\begin{lemma}\label{chainrule}
Suppose $k \in L^1_{\rm loc}(X)$ has an upper gradient $g_k \in L^1_{\rm loc}(X)$ and $f:\R \rightarrow \R$ has a bounded and Lipschitz continuous derivative. Then
$$g_{f \circ k}(x) \leq |f'(k(x))|g_k(x).$$ 
\end{lemma}
\begin{proof}
Assume that $C>0$ is such that $|f'(x)| \leq C$ and $|f'(x)-f'(y)| \leq C|x-y|$ for all $x,y \in \R$.
It is enough to prove that for any given $\eta \in \WR(\X)$ with $b_{\eta} >0$ we have
\begin{equation}\label{chainest}
\limsup_{s \rightarrow 0} \left| \frac{1}{s} \left( \int f(k(\cdot)) \, d\eta(s) - \int f(k(\cdot)) d \, \eta(0)\right)\right| \leq \int |f'(k(\cdot))|g_k \, d \eta(0).
\end{equation}
To do so let $K = \cup_{s \in [0,b_{\eta}]} \supp(\eta(s))$ which is a compact subset of $X$ and let $\vae >0$ be fixed. Also fix a continuous function $\tilde{g}_k$ such that 
$$\int_K |g_k - \tilde{g}_K| d \mu \leq \vae.$$
By absolute continuity we may furthermore choose $\delta>0$ such that for any $A \subset K$ with $\mu(A)<\delta$ we have $\int_A g_k d\mu \leq \vae/2$. It is then easy to see that with 
$$\delta' = \frac{\delta \vae }{2\int_K g_k d\mu}$$
we have
$$\int  g_k \phi d\mu \leq \vae$$
for any $\phi$ with compact support in $K$ and values in $[0,1]$ such that
$$\int \phi d\mu \leq \delta'.$$
I.e. we have that 
$$\int g_k d\eta \leq \vae$$
for any element $\eta \in \X$ with support in $K$ and $\|\eta\|\leq \delta'$.
By Lusin's theorem we may choose $K' \subset K$ compact such that $k|_{K'}$ is continuous and $\mu(K\setminus K') < \delta'/2$. Let $\psi$ be a common modulus of continuity for $\tilde{g}_k$ and $k|_{K'}$. For $s \in (0,b_{\eta}]$ fixed we may then cover $K$ by balls $B_1,B_2,\ldots,B_m$ of radius at most $s$. If we introduce the sets 
$$A_{i,j}=\{x \in X: (i-1)s^2 \leq k(x) <is^2\} \cap B_j \cap K' \textrm{ for } i \in \mathbb{Z}, j=1,2,\ldots,m,$$
$$ A_{i,m+1}=\{x \in X: (i-1)s^2 \leq k(x) <is^2\} \setminus K',$$
and the measures 
$$\mu_{i,j} = \mu|_{A_{i,j}}, \quad i \in  \mathbb{Z}, j \in \{1,2,\ldots,m+1\},$$
then clearly
$$\mu = \sum_{i,j} \mu_{i,j}.$$

If we apply Corollary \ref{split3} to $\eta$, $t_0=0$, $t_1=s$ and $\MM = \mathbb{Z} \times \{1,2,\ldots,m+1\}$ with $\mu_{i,j}$ as above we get that there are curves $\nu_j \in \WR(\X)$ such that
\begin{itemize}
\item for each $j \in \N$ there is $(i_0,j_0),(i_s,j_s) \in \MM$ such that $\nu_j(0) \leq \mu_{i_0,j_0}$ and $\nu_j(s) \leq \mu_{i_s,j_s}$,
\item $\eta(t) = \sum_{j=1}^{\infty} \nu_j(t)$ for every $t \in [0,b_{\eta}]$.
\end{itemize}

So on carriers for the measures $\nu_j(0)$ and $\nu_j(s)$ we have that the oscillation of $k$ is no more than $s^2$, and hence differs from its mean value with respect to these measures by at most $s^2$.
(Also note that we do not make any claims of this nature for the values between $0$ and $s$, but only for these end-points).

Let $I =\{j:\nu_j \ne 0\}$, $I_1=\{j \in I: j_0 =m+1\} \cup\{j \in I: j_s=m+1\}$ and $I_2 = I \setminus I_1$. Note that $\sum_{j \in I_1} ||\nu_j|| \leq 2\mu(K \setminus K') \leq \delta'$. Now we get for suitable $\tau_{j,s}(x)$ between $k(x)$ and $F_k(\nu_{j}(s))$ and $\theta_{j,s}$ between $F_k(\nu_{j}(s))$ and $F_k(\nu_{j}(0))$ 
\begin{align*}
& \frac{1}{s} \left|\int f(k(\cdot)) \, d\eta(s) - \int f(k(\cdot)) \, d\eta(0)\right| \\
&= \frac{1}{s} \left| \sum_{j \in I} \left(\int f(k(\cdot)) \, d\nu_{j}(s) - \int f(k(\cdot))\,  d\nu_{j}(0)\right)\right| \\
&= \frac{1}{s} \left| \sum_{j \in I} \left(\int\left( 
f(F_k(\nu_{j}(s))) + f'(\tau_{j,s}(\cdot))(k(\cdot)-F_k(\nu_{j}(s)))\right) \, d\nu_{j}(s)  \right.\right. \\
&- \left.\left.\int\left( f(F_k(\nu_{j}(0))) + f'(\tau_{j,0}(\cdot))(k(\cdot)-F_k(\nu_{j}(0)))\right) \, d\nu_{j}(0)\right)\right|  \\
&\leq \frac{1}{s}\left|\sum_{j \in I} f'(\theta_{j,s})(F_k(\nu_{j}(s))-F_k(\nu_{j}(0)))\|\nu_{j}(0)\|\right|  \\
&+ \sum_{j \in I}\frac{C}{s}\left|\int (k-F_k(\nu_{j}(s))) \, d\nu_{j}(s)\right| +  \sum_{j \in I}\frac{C}{s}\left|\int (k-F_k(\nu_{j}(0))) \, d\nu_{j}(0)\right|  \\
&\leq \frac{1}{s}\sum_{j \in I} |f'(\theta_{j,s})| \int_0^s \int g_k \, d\nu_{j}(t)dt + 2C\mu(K)s  \\
&\leq \frac{1}{s}\sum_{j \in I_1} |f'(\theta_{j,s})| \int_0^s \int g_k \, d\nu_{j}(t)dt + \frac{1}{s}\sum_{j \in I_2} |f'(\theta_{j,s})| \int_0^s \int g_k \, d\nu_{j}(t)dt + 2C\mu(K)s\\
& \leq \frac{1}{s} C \int_0^s (\int g_k \, d\left(\sum_{j \in I_1} \nu_{j}(t)\right))dt  + \frac{1}{s}\sum_{j \in I_2} |f'(\theta_{j,s})| \int_0^s \int \tilde{g}_k \, d\nu_{j}(t)dt +C(2\mu(K)s+\vae)\\
&\leq C\vae + \sum_{j \in I_2} |f'(\theta_{j,s})| \left| \frac{1}{s} \int_0^s \int \tilde{g}_k \, d\nu_j(t) dt - \int \tilde{g}_k \, d\nu_j(s)\right| + \sum_{j \in I_2} \int |f'(k(\cdot))| \tilde{g}_k \, d\nu_j(s) \\
&+ \sum_{j \in I_2} \int |f'(\theta_{j,s}) - f'(k(\cdot))|\tilde{g}_k \, d\nu_j(s) + C(2\mu(K)s+\vae)  \\
&\leq C\vae +2C\psi(4s)\mu(K)+ \int |f'(k(\cdot))| g_k d\eta(s) + C\vae  \\
&+ \sum_{j \in I_2} \int C|\theta_{j,s} - k(\cdot)|\tilde{g}_k \, d\nu_j(s)+C(2\mu(K)s+\vae)\\
&\leq \int |f'(k(\cdot))| g_k d\eta(s)  + 2C\psi(4s)\mu(K) + 4C \psi(4s) \int_K \tilde{g}_k\, d\mu +C(2\mu(K)s+3\vae).
\end{align*}
Above we used the fact that for $j \in I_2$ the diameter of $\cup_{0 \leq t \leq s}\supp(\eta_j(t))$ is at most $4s$, and hence $\tilde{g}_k$ differs from its mean value by at most $\psi(4s)$, and similarly we get
$$|\theta_{j,s}-k| \leq 4\psi(4s),$$
since both measures $\nu_j(0)$ and $\nu_j(s)$ are supported by $K'$.
If we let $s$ go to zero the last expression in the estimate above goes to 
$$\int |f'(k(\cdot))| g_k d\eta(0)  + 3C\vae.$$
 And since $\vae>0$ is arbitrary this implies that (\ref{chainest}) holds as was to be proved.
\end{proof}
\begin{theorem}\label{lattice1}
If $k \in L^1_{\rm loc}(X)$ has an upper gradient $g_k \in L^1_{\rm loc}(X)$ and $c \in \R,$ then  $$g_{k \wedge c} = g_k \chi_{\{k<c\}}$$
and
$$g_{k \vee c} = g_k \chi_{\{k >c\}}.$$
\end{theorem}
\begin{proof}
To prove the inequalities
\begin{equation}\label{ineqlat}
g_{k \wedge c} \leq g_k \chi_{\{k<c\}} \textrm{ and }g_{k \vee c} \leq g_k \chi_{\{k >c\}}.
\end{equation}
it is easy to see that we may without loss assume $c=0$ and that we look at the case $k^+ = k \vee 0$.
For each $\vae >0$ we may introduce the functions 
$$f_{\vae}(x) = \left\{ \begin{array}{ll} \sqrt{x^2+\vae^2} - \vae & x>0\\ 0 & x \leq 0.\end{array}\right.$$
Applying the previous lemma we get 
$$\left | \int f_{\vae}(k(\cdot)) \, d\eta(s) - \int f_{\vae}(k(\cdot)) \, d\eta(0)\right| \leq \int_0^s \int f'_{\vae}(k(\cdot))g_k \, d\eta(t) \, dt.$$
Since the left hand side converges to 
$$\left | \int k^+ \, d\eta(s) - \int k^+ \, d\eta(0)\right|,$$
and the right hand side to 
$$ \int_0^s \int\chi_{\{k>0\}}g_k \, d\eta(t) \, dt$$
we get the desired estimate.

However we also have
$$g_k = g_{k\wedge c + k \vee c -c} \leq g_{k \wedge c} + g_{k \vee c} + g_c,$$
but $g_c=0$  trivially, and therefore we get from the inequalities (\ref{ineqlat})
$$g_k \leq g_{k} \chi_{\{k<c\}} + g_{k \vee c},$$
or which amounts to the same thing
$$g_{k \vee c} \geq g_k\chi_{\{k \geq c\}} \geq  g_k\chi_{\{k > c\}}.$$
Similarly we get
$$g_{k \wedge c} \geq g_k \chi_{\{k<c\}}.$$

\end{proof}

\begin{proposition} \label{ugfb2}
Suppose $f,h \in L^1_{\rm loc}(X)$ are such that $g_f,g_h \in L^1_{\rm loc}(X)$ then
$$g_{fh} \leq |f|g_h + |h|g_f.$$ 
\end{proposition}
\begin{proof}
We first of all reduce the problem to the case when $f$ and $h$ are bounded. To do so suppose the statement is true for bounded functions. Then we have (using Theorem \ref{lattice1}) for $c_1 \leq c_2 \in \R$
\begin{align*}
&\left| \int ((f \vee c_1) \wedge c_2)((h \vee c_1) \wedge c_2) \, d\eta(s) -  \int ((f \vee c_1) \wedge c_2)((h \vee c_1) \wedge c_2) \, d\eta(0) \right| \\
&\leq \int_0^s\left( \int (|(f \vee c_1)\wedge c_2|g_{(h\vee c_1)\wedge c_2)} + |(h\vee c_1) \wedge c_2|g_{(f\vee c_1)\wedge c_2}) \, d\eta(t)\right) \, dt\\
&\leq \int_0^s \left( \int (|f|g_h + |h|g_f) \, d\eta(t)\right)\, dt.
\end{align*}
Since we may then take the limit as first $c_2$ goes to infinity and then $c_1$ goes to minus infinity and use monotone convergence we hence get the statement we need for general $f,h$. So from now on we assume that there is a constant $C$ such that $|f|\leq C, |h|\leq C$ everywhere.

Let $\eta \in \WR(\X)$ with $b_{\eta}>0$ and assume that all the supports of the measures $\eta(s)$ are contained in the compact set $K$. Cover $K$ by finitely many balls $B_1,B_2,\ldots,B_m$ with radii at most $s^2$. Now we partition $X$ (up to a set of measure zero) as follows.
Let
$$\MM=\{i,j,r : i,j \in \mathbb{Z} \textrm{ and } r \in \{1,2,\ldots,m\}\}.$$
For $s$ fixed and $i,j,r \in \MM$ we let
$$A_{i,j,r} = \left\{x \in X : (i-1)s^2 < f(x) \leq is^2,\, (j-1)s^2 < h(x) \leq js^2\right\} \cap B_r$$
and define
$$\mu_{i,j,r}=\mu|_{A_{i,j,r}}.$$
Let $s \in (0,b_{\eta}]$. If we apply Corollary \ref{split3} to this decomposition of $\mu$ we see that there are curves $\nu_k \in \WR(\X)$ such that $\eta = \sum_{k=1}^{\infty} \nu_k$ and for each $k$ there are $i_0,j_0,r_0,i_s,j_s,r_s$ such that $\nu_k(0) \leq \mu_{i_0,j_0,r_0}$ and $\nu_k(s) \leq \mu_{i_s,j_s,r_s}$. So on carriers for the measures $\nu_i(0)$ and $\nu_i(s)$ respectively the oscillation of both $f$ and $h$ are at most $s^2$. Furthermore the supports of all of the measures $\nu_i(0)$ and $\nu_i(s)$  has diameter at most $2s^2$ . Let $I=\{i:\nu_i \ne 0\}$. Now we get (using that $||\nu_i||$ and $||\eta||$ are constant)
\begin{align*}
&  \left| \int fh \, d\eta(s) - \int fh \, d\eta(0)\right| = \left| \sum_{i \in I} \left(\int fh \, d\nu_i(s) - \int fh \, d\nu_i(0)\right)\right|  \\
 &\leq  \left| \sum_{i \in I} 
 \int \left(f - \frac{1}{||\nu_i(s)||}\int f \, d\nu_i(s)\right)h d\nu_i(s)\right| \\
 &+ \left| \sum_{i \in I} 
  \frac{1}{||\nu_i(0)||}\left(\int f \, d \nu_i(s) - \int f \, d\nu_i(0)\right) \int h \, d\nu_i(0)\right|\\
  &+ \left| \sum_{i \in I} 
  \frac{1}{||\nu_i(0)||}\left(\int h \, d \nu_i(s) - \int h \, d\nu_i(0)\right) \int f \, d\nu_i(s)\right| \\
  &+ \left| \sum_{i \in I} 
 \int \left(h - \frac{1}{||\nu_i(0)||}\int h \, d\nu_i(0)\right)f \, d\nu_i(0)\right| \\
 &\leq    \sum_{i \in I} 
 \int s^2 |h| \, d\nu_i(s) +   \sum_{i \in I} 
  \int_0^s r_{F_f}(\nu_i(t)) \, dt \int |h| \, d\nu_i(0)\\
  &+   \sum_{i \in I} 
  \int_0^s r_{F_h}(\nu_i(t)) \, dt \int |f| \, d\nu_i(s) +   \sum_{i \in I} 
 \int s^2 |f| \, d\nu_i(0) \\
 &=     
 \int s^2 |h| \, d\eta(s) +  \sum_{i \in I} 
  \int_0^s \left(\frac{1}{||\nu_i(0)||} \int g_f \, d \nu_i(t) \right) \, dt \int |h| \, d\nu_i(0)\\
  &+  \sum_{i \in I} 
  \int_0^s \left(\frac{1}{||\nu_i(0)||} \int g_h \, d \nu_i(t) \, dt\right) \int |f| \, d\nu_i(s) +  
 \int s^2 |f| \, d\eta(0).
\end{align*}
Let $\vae>0$ and choose non-negative continuous functions $\tilde{g}_f$ and $\tilde{g}_h$  with common modulus of continuity $\psi$ on $K$ such that
$$\int_K |g_f - \tilde{g}_f| \, d\mu  \leq \vae,\quad \int_K |g_h - \tilde{g}_h| \, d\mu  \leq \vae.$$
Then we get
\begin{align*}
&  \limsup_{s \rightarrow 0^+} \frac{1}{s}\left| \int fh \, d\eta(s) - \int fh \, d\eta(0)\right|  \\
&\leq \limsup_{s \rightarrow 0^+} \frac{1}{s}\left(   
 \int s^2 |h| \, d\eta(s) +   \sum_{i \in I} s
  \int_0^s \left(\frac{1}{||\nu_i(0)||} \int g_f \, d \nu_i(t) \right) \, dt \int |h| \, d\nu_i(0)\right.\\
  &+\left.  \sum_{i \in I} 
  \int_0^s \left(\frac{1}{||\nu_i(0)||} \int g_h \, d \nu_i(t) \right) \, dt \int |f| \, d\nu_i(s) +  
 \int s^2 |f| \, d\eta(0)\right)\\
 &\leq \limsup_{s \rightarrow 0^+} \frac{1}{s} \sum_{i \in I} 
  \int_0^s \left( \int |g_f-\tilde{g}_f| d \nu_i(t) \right) \, dt \frac{1}{||\nu_i(0)||}\int |h| \, d\nu_i(0) \\
  &+ \limsup_{s \rightarrow 0^+} \frac{1}{s} \sum_{i \in I}\int_0^s \left(\frac{1}{||\nu_i(0)||}\int \tilde{g}_f \, d\nu_i(t)\right) \, dt \int |h| \, d\nu_i(0)+\\
   & \limsup_{s \rightarrow 0^+} \frac{1}{s} \sum_{i \in I} 
  \int_0^s \left( \int |g_h-\tilde{g}_h| \,  d \nu_i(t) \right)\, dt \frac{1}{||\nu_i(0)||}\int |f| \, d\nu_i(s) \\
  &+ \limsup_{s \rightarrow 0^+} \frac{1}{s} \sum_{i \in I}\int_0^s \left(\frac{1}{||\nu_i(0)||}\int \tilde{g}_h \, d\nu_i(t)\right) \, dt \int |f| \, d\nu_i(s).
 \end{align*}
From our assumptions we get that
 $$\limsup_{s \rightarrow 0^+} \frac{1}{s} \sum_{i \in I} 
  \int_0^s \left( \int |g_f-\tilde{g}_f| \, d \nu_i(t) \right) \, dt \frac{1}{||\nu_i(0)||}\int |h| \, d\nu_i(0) \leq C\vae$$
and
$$\limsup_{s \rightarrow 0^+} \frac{1}{s} \sum_{i \in I} 
  \int_0^s \left( \int |g_h-\tilde{g}_h| \, d \nu_i(t) \right) \, dt \frac{1}{||\nu_i(0)||}\int |f| \, d\nu_i(s) \leq C\vae.$$
Also, since $\cup_{t\in [0,s]}\supp(\nu_i(t))$ has diameter at most $2(s+s^2)$, we get
\begin{align*}
&\frac{1}{s} \sum_{i \in I}\int_0^s \left(\frac{1}{||\nu_i(0)||}\int \tilde{g}_h \, d\nu_i(t)\right) \, dt \int |f| \, d\nu_i(s)\\
&\leq \sum_{i \in I } \int \left| \frac{1}{\|\nu_i(0)\|}\int_0^s \tilde{g}_h d\nu_i(t) - \tilde{g}_h \right| |f|d\nu_i(s) + \sum_{i \in I} \int \tilde{g}_h|f|d\nu_i(s) \\
&\leq C\|\eta(s)\|\psi(2(s+s^2)) +C\vae + \int g_h |f|\, d\eta(s).
\end{align*}
Hence 
$$\limsup_{s \rightarrow 0}  \sum_{i \in I}\int_0^s \left(\frac{1}{||\nu_i(0)||}\int \tilde{g}_h \, d\nu_i(t)\right) \, dt \int |f| \, d\nu_i(s) \leq \int g_h |f| \,d\eta(0) + C\vae.$$
Similarly we get
$$\limsup_{s \rightarrow 0}  \sum_{i \in I}\int_0^s \left(\frac{1}{||\nu_i(0)||}\int \tilde{g}_f \, d\nu_i(t)\right) \, dt \int |h| \, d\nu_i(s) \leq \int g_f |h| \,d\eta(0) + C\vae.$$
Summing up we get
$$\limsup_{s \rightarrow 0^+} \frac{1}{s}\left| \int fh \, d\eta(s) - \int fh \, d\eta(0)\right| \leq 
\int (g_h |f|+g_f |h|) \,d\eta(0) +4C\vae.$$
And since $\vae>0$ is arbitrary we get the statement.
\end{proof}
\section{The space $\S^{1,p}(\X)$}\label{SOBOLEV}
We now define
$$\|F\|_{\S^{1,p}(\X)} : = \left(\|F\|_{\L^p(\X)}^p + \|r_F\|_{\L^p(\X)}^{p}\right)^{1/p},$$
$$\S^{1,p}(\X) = \{ F \in \L^p(\X) : \|F\|_{\S^{1,p}(\X)} < \infty\}.$$
It is easily verified that $\|\cdot\|_{\S^{1,p}(\X)}$ is a norm on $\S^{1,p}(\X)$.
\begin{remark}
In case $F \in \S^{1,p}(\X)$, then by assumption $r_F \in \L^p(\X)$, and hence $\breve{r}_F \in \L^p(\X)$ with the same norm. In particular $\breve{r}_F(\eta)$ is  finite for every $\eta \in \M$. Moreover if $\eta \in \WR(\X)$ then, since $\|\eta(s)\|$ is constant, $\sup_{s \in [0,b_{\eta}]} \breve{r}_F(\eta(s)) \leq \|r_F\|_{\L^p(\X)}/\|\eta(0)\|=k < \infty$, so $|F(\eta(s)) - F(\eta(t))| \leq k|s-t|$ for all $s,t \in [0,b_{\eta}]$, and hence $F(\eta(s))$ is Lipschitz continuous in $s$.
\end{remark}
\begin{theorem} \label{banach1}
$\S^{1,p}(\X)$ is a Banach space. Furthermore, if $F_n \rightarrow F$ in $\S^{1,p}(\X)$, then $r_{F_n} \rightarrow r_F$ in $\L^p(\X)$.
\end{theorem}
\begin{proof}
Suppose $F_j \in \S^{1,p}(\X)$ is a Cauchy sequence. By passing to a subsequence we may assume that
$\|F_{j+1}-F_j\|_{\L^p (\X)} + \|r_{(F_{j+1}-F_j)}\|_{\L^p(\X)} \leq 2^{-j}.$ By definition
$$ \|F_j - F_i\|_{\L^p(\X)} \leq \|F_j-F_i\|_{\S^{1,p}(\X)}$$
and, since $|r_{F_j} - r_{F_i}| \leq r_{(F_j-F_i)}$ according to Theorem \ref{ugbasic},
$$ \|r_{F_j} - r_{F_i}\|_{L^p(\X)} \leq \|\breve{r}_{(F_j - F_i)}\|_{\L^p(\X)} \leq \|F_j-F_i\|_{\S^{1,p}(\X)}.$$
In particular $F_j$ is a Cauchy sequence in $\L^p(\X)$, and hence converges in this space to some $F \in \L^p(\X)$. 

For any $k \in \N$ we have 
$$F-F_k = \sum_{j=k}^{\infty} (F_{j+1}-F_j),$$
so according to Theorem \ref{countadd} we know that
$$r_{(F-F_k)} \leq \breve{g} \textrm{ where } g=\sum_{j=k}^{\infty} \breve{r}_{(F_{j+1}-F_j)}.$$
But then
$$\|r_{F} - r_{F_k}\|_{\L^p(\X)} \leq\|r_{(F-F_k)}\|_{\L^p(\X)} \leq \|\breve{g}\|_{\L^p(\X)} = \|g\|_{\L^p(\X)} \leq \sum_{j=k}^{\infty} \|\breve{r}_{(F_{j+1}-F_j)}\|_{\L^p(\X)} \leq 2^{1-k}.$$
If we add all this together we see that indeed $F_k$ converges to $F$ and $r_{F_k}$ converges to $r_F$ in $S^{1,p}(\X)$ as $k \rightarrow \infty$.
\end{proof}
\begin{theorem}
If $F,G \in \S^{1,p}(\X)$ then $F \vee G,F \wedge G \in \S^{1,p}(\X)$.
\end{theorem}
\begin{proof}
Since $|F \vee G| \leq |F| \vee |G| \leq |F| + |G|$ and $|F \wedge G| \leq |F| \wedge |G| \leq |F|+|G|$ we see that $F \vee G,F\wedge G \in \L^p(\X)$. By Lemma \ref{lattice} we get that $\breve{r}_{F \vee G}, \breve{r}_{F \wedge G} \in \L^p(\X)$, and hence the theorem follows.
\end{proof}
\section{The spaces $S^{1,p}(\X)$ and $S^{1,p}(X)$}\label{Sobolev2}
We let
$$S^{1,p}(\X) = \{ F \in L^p(\X) : \|F\|_{\S^{1,p}(\X)} < \infty\},$$
and
$$S^{1,p}(X) =\{f \in L^p(X): F_f \in S^{1,p}(\X)\}.$$
In the latter case we as usual identify elements which are equal a.e., and we also define the norm on this space 
$$\|f\|_{S^{1,p}(X)} = \|F_f\|_{\S^{1,p}(\X)}.$$
\begin{remark}
There is an obvious $1-1$ correspondence, according to our previous results, between these two spaces. $F$ belongs to $S^{1,p}(\X)$ if and only if there is $f \in S^{1,p}(X)$ such that $F=F_f$, and then
$$\|f\|_{S^{1,p}(X)}=\|F_f\|_{\S^{1,p}(\X}) = \left( \|f\|_{L^p(X)}^p + \|g_f\|_{L^p(X)}^p\right)^{1/p}  .$$
\end{remark}
\begin{theorem} \label{banach2}
$S^{1,p}(\X)$ and $S^{1,p}(X)$ are Banach spaces. Furthermore if $f_n \rightarrow f$ in $S^{1,p}(X)$ then $g_{f_n} \rightarrow g_f$ in $L^p(X)$.
\end{theorem}
\begin{proof}
It is enough, according to Theorem \ref{banach1}, to note that in case $F_j \in L^p(\X) \cap \S^{1,p}(\X)$ and $F_j \rightarrow F$ in $\S^{1,p}(\X)$ then it is clear that also $F \in L^p(\X)$ and hence the theorem follows.
\end{proof}

\begin{theorem}\label{latticeS}
Suppose $f \in S^{1,p}(X)$, then $f^+,f^-$ and $|f|$ also belongs to $S^{1,p}(X)$.
\end{theorem}
\begin{proof}
This is an immediate consequence of Theorem \ref{lattice1}.
\end{proof}
\section{$S^{1,p}(X)$ for open subsets $X$ of $\R^n$} \label{RN}
Let $X \subset \R^n$ be open, $d_X$ denote the Euclidean metric and let $\mu$ be Lebesgue measure on $X$. 

If $f : X \rightarrow \R$ then we denote the distributional gradient of $f$ by $\nabla f$. In case both $f$ and $|\nabla f|$ belongs to $L^p(X)$ then we say that $f$ belongs to the classical Sobolev space $H^{1,p}(X)$.
We give $H^{1,p}(X)$ the norm 
$$\|f\|_{H^{1,p}(X)} = (\|f\|_{L^p(X)}^p + \| |\nabla f|\|_{L^p(X)}^p)^{1/p}.$$
Our aim is to show that in this context the spaces $S^{1,p}(X)$ and $H^{1,p}(X)$ coincide.
\begin{lemma}\label{gnabla}
If $f \in C^2(X)$ then $g_f=|\nabla f|$.
\end{lemma}
\begin{proof}
We start with the inequality $|\nabla f| \leq g_f$. To prove this it is necessary and sufficient to prove that for any unit vector $\overline{e}$ we have that the directional derivative $\partial_{\overline{e}}f$ satisfies $|\partial_{\overline{e}}f| \leq g_f$ pointwise a.e.

Now suppose $B(x,r) \subset \subset X$. Then we get
$$ \int_{B(x,r)} |\partial_{\overline{e}}f(y)| \, d\mu(y) = \lim_{s \rightarrow 0^+}   \int_{B(x,r)} \left| \frac{f(y+s\overline{e}) - f(y)}{s}\right| \, d \mu(y).$$

In case $\partial_{\overline{e}}f(x) =0$ then there is nothing to prove. On the other hand, in case it is non-zero we may choose $\vae,\delta >0$ so small that $f(\cdot+s\overline{e})-f(\cdot)$ does not change sign in $B(x,\delta)$ for $s \in [0,\vae]$ say. Then we get for any $ r \leq \delta$
\begin{align*}
& \lim_{s \rightarrow 0^+}  \int_{B(x,r)} \left| \frac{f(y+s\overline{e}) - f(y)}{s}\right| \, d \mu(y) \\
&= \lim_{s \rightarrow 0^+} \left| \int_{B(x,r)}  \frac{f(y+s\overline{e}) - f(y)}{s} \, d \mu(y)\right| \\
&= \lim_{s \rightarrow 0^+} \left|\frac{1}{s} \left(\int_{B(x+s\overline{e},r)} f \, d\mu - \int_{B(x,r)} f \, d\mu\right) \right| \leq  \int_{B(x,r)} g_f \, d\mu.
\end{align*}
Hence for every $x \in X$ there is $\delta >0$ such that for every $r \leq \delta$
$$ \int_{B(x,r)} |\partial_{\overline{e}}f(x)| \, d\mu(x) \leq \int_{B(x,r)} g_f \, d\mu,$$
so the stated inequality holds.

To prove the opposite inequality suppose $\eta \in \WR(\X)$. Let $s>0$. We may cover $X$ by a countable disjoint family of Borel sets $A_j$, $j=1,2,\ldots$ such that each has diameter at most $s^2$ and such that the oscillation of $f$ over such a set is at most $s^2$ (similarly to the construction made in the proof of Proposition \ref{ugfb2}). If we define $\eta_j(0)= \eta(0)|_{A_j}$ we may according to Proposition \ref{split2} get a decomposition of $\eta(t)=\sum_{j=1}^{\infty} \eta_j(t)$ valid for $t \in [0,b_{\eta}]$, where each $\eta_j \in \WR(\M)$. Now we get if we fix $x_j \in A_j$ (using that $|x-x_j|\leq s^2$ if $x \in {\rm supp}(\eta_j(0))$, and $|x-x_j| \leq s+s^2$ if $x \in {\rm supp}(\eta_j(s))$)
\begin{align*}
& \left| \int f \, d\eta(s)-\int f \, d\eta(0) \right| =  \left|\sum_{j=1}^{\infty} \left( \int f \, d\eta_j(s)-\int f \, d\eta_j(0) \right)\right|  \\
&\leq \left|\sum_{j=1}^{\infty}  \int (f(\cdot)-f(x_j)) \, d\eta_j(s)\right| + \left|\sum_{j=1}^{\infty}  \int (f(\cdot)-f(x_j)) \, d\eta_j(0)\right| \\
&\leq  \left|\sum_{j=1}^{\infty}  \int \nabla f \cdot (x_j-\cdot) \, d\eta_j(s)\right| + o(s) + s^2  \\
&\leq \left(\sum_{j=1}^{\infty}  \int |\nabla f| \, d\eta_j(s)\right)s + o(s) + s^2 =   \left(\int |\nabla f|  \, d\eta(s)\right)s + o(s)+ s^2.
\end{align*}
Since this last expression is not dependent on the particular decomposition and  
$$\lim_{s \rightarrow 0} \frac{1}{s}\left(\left(\int |\nabla f|  \, d\eta(s)\right)s + o(s)+ s^2\right) =  \int |\nabla f| \, d \eta(0)$$
we see that we also have $g_f \leq |\nabla f|$ according to Proposition \ref{fundprop}.
\end{proof}

\begin{theorem}
Assume that $1 \leq p<\infty$. Then $S^{1,p}(X)=H^{1,p}(X)$ and the norms are the same. Furthermore $g_f =|\nabla f|$ for every $f \in H^{1,p}(X).$
\end{theorem}
\begin{proof}
Assume that $f \in H^{1,p}(X)$, and choose a sequence $f_n \in C^{\infty}(X)$ such that $f_n$ converges to $f$ in $H^{1,p}(X)$. Then from Lemma \ref{gnabla} we get
$$\|f_n-f_m\|_{S^{1,p}(X)}=\|f_n-f_m\|_{H^{1,p}(X)}.$$
Hence we see that $f_n$ is a Cauchy sequence in $S^{1,p}(X)$, and since it converges to $f$ in $L^p(X)$ it follows that $f \in S^{1,p}(X)$.  

Furthermore we get for any $\eta \in \WR(\M)$ and $n \in \N$
$$\left| \int f_n  \, d\eta(s) - \int f_n \, d \eta(0)\right| \leq \int_0^s \int |\nabla f_n|\, d\eta(t) \, dt.$$
If we let $n \rightarrow \infty$ we see that
$$\left| \int f \, d\eta(s) - \int f \, d \eta(0)\right| \leq \int_0^s \int |\nabla f| \, d\eta(t) \, dt.$$
Hence $g_f \leq |\nabla f|$.

Conversely, suppose $f \in S^{1,p}(X)$, $\phi$ is Lipschitz continuous with compact support in $X$ and $0 \leq \phi \leq 1$. Then for a unit vector $\overline{e}$ we know that $\eta(t)=\phi(\cdot+t\overline{e}) \mu$  belongs to $\WR(\M)$. Hence
\begin{align*}
&\left|\int f \partial_{\overline{e}} \phi \, d\mu \right| = \lim_{s \rightarrow 0} \left| \int f \frac{\phi(\cdot+s\overline{e})-\phi}{s} \, d\mu \right|\\
&= \lim_{s \rightarrow 0} \left| \frac{1}{s}\left( \int f \, d\eta(s) - \int f \, d\eta(0)\right) \right| \leq \int g_f \, d\eta(0) &\\& =\int g_f \phi \, d\mu.
\end{align*}
However, if $\phi=c_1 \phi_1+c_2\phi_2$ where $\phi_i$ are Lipschitz continuous with compact support and $0 \leq \phi_i \leq 1$ and $c_1,c_2 \in \R$, then
\begin{align*}
&\left|\int f \partial_{\overline{e}} \phi \, d\mu \right|  \\
&\leq|c_1|\left|\int f \partial_{\overline{e}} \phi_1 \, d\mu \right| + 
|c_2|\left|\int f \partial_{\overline{e}} \phi_2 \, d\mu \right| \\
&\leq|c_1|\int g_f \phi_1 \, d\mu + |c_2|\int g_f \phi_2 \, d\mu.
\end{align*}
In particular this implies that we have, for all $\phi$ which are Lipschitz continuous with compact support in $X$,
$$\left|\int f \partial_{\overline{e}} \phi \, d\mu \right| \leq \int g_f |\phi| \, d\mu.$$
Therefore we see that the distribution $\partial_{\overline{e}} f$ has order zero, it is absolutely continuous with respect to $\mu$, and it has a Radon-Nikodym derivative whose absolute value is dominated by $g_f$. Hence $\partial_{\overline{e}} f$ has a representative which belongs to $L^p(X)$ for each $\overline{e}$, and also this representative satisfies 
$|\partial_{\overline{e}} f| \leq g_f$ $\mu$-a.e. Hence $|\nabla f| \leq g_f$ and the proof is done.
\end{proof}
\section{Some final remarks}
In this section we wish to make some remarks concerning certain choices and open questions related to this article.\smallskip\newline
{\bf Choice of space $\X$ and metric $d_{\M}$:}\smallskip\newline
It is not self-evident that the choice of metric and space $\X$ are optimal for this type of construction. For instance one could have considered instead of the bound $d\nu/d\mu \leq 1$ perhaps that we should only have a bounded Radon-Nikodym derivative with respect to $\mu$. We wanted also to have a theory where the upper gradients did not depend on the integrability exponent $p$ (indeed the definition of $r_F$ makes no assumption about integrability). Otherwise one could perhaps consider spaces which depends on $p$, and perhaps also relax the condition to have compact support of the mesures (e.g. simply assuming that the Radon-Nikodym derivative lies in $L^q(X)$ where $q$ is the dual exponent). So this is one possible area that could be worth investigating.\smallskip\newline
{\bf Relation to the Wasserstein metric:}\smallskip\newline 
It would in many respects be natural to look at  $$\widehat{\X} =\{\eta/\|\eta\|: \eta \in \X\}$$
rather than $\X$ itself, in particular considering the formula for $F_f$. Then $\widehat{\X}$ is a space of probability measures, and one could introduce a metric on this set. Let us for an element $\eta \in \X$ define $\widehat{\eta}= \frac{1}{\|\eta\|}\eta \in \widehat{\X}$.
We recall that the Wasserstein $1$-metric $W_1$ can be defined as follows:
$$W_1(\widehat{\eta},\widehat{\nu})= \sup \{\int f \, d(\widehat{\eta} - \widehat{\nu}) : {\rm Lip}(f) \leq 1\}.$$
This is in some sense the classical mass transport metric, and
our metric will be a type of mass transport metric on $\X$, but with the slight difference that in general it allows for not just relocation of mass but also controlled change in total mass. 

It is  clear from Proposition \ref{lipgrad} that along a curve $\eta \in \RR(\X)$ we have
$d_{\M}(\eta(s),\eta(t)) = |s-t| \geq W_1(\widehat{\eta}(s),\widehat{\eta}(t)).$

A couple of properties of $d_{\M}$ that we use extensively is first that in case $\eta$ and $\nu$ are close in the metric $d_{\M}$ then so are their total mass (in $\widehat{\X}$ of-course all measures are probability measures, but we need control of the size compared to $\mu$ for our construction).
The point is that if $\|\eta\|=\|\nu\|$ then
$$\frac{1}{\|\eta\|} \int f \, d\eta - \frac{1}{\|\nu\|} \int f \, d\nu = \frac{1}{\|\eta\|} \left(\int f \, d\eta - \int f \, d\nu\right).$$
It may be worthwhile to note that if we for instance work with $X=\R$ and $\mu$ Lebesgue measure, and we were to use a metric such that $\eta(t) = \mu|_{[0,1+t]}$ belonged to $\WR(\X)$, then for any continuous function $f$ we would have
$$\lim_{s \rightarrow 0^+} \left|\frac{1}{s}\left(\int f \, d\eta(s) - \int f \, d\eta(0) \right)\right| = f(1).$$
This sort of phenomenon is obviously something we need to avoid, and hence some kind of control of the total mass of $\eta$ compared to $\mu$ seems necessary.
 
Furthermore the decomposition results such as that in Proposition \ref{split2} was also crucial to us.
Let us here compare the situation with $W_1$ by studying some curves on the real line. So let $X=\R$ with Lebesgue measure $\mu$. If we start by studying $\eta(t)=\widehat{\eta}(t) =\mu|_{[t,1+t]}$, then it is easy to see that indeed 
$$d_{M}(\eta(t),\eta(0)) =W_1(\widehat{\eta} (t),\widehat{\eta} (0))=t.$$
If we instead look at $\nu(t)=\widehat{\nu}(t)= \mu|_{[0,1/2]}+\mu_{|1/2+t,1+t]}$, then 
$$d_{\M}(\eta(t),\eta(0))=t,$$
but
$$W_1(\widehat{\eta}(s),\widehat{\eta}(0))= t/2.$$
It certainly would be very interesting to investigate if it is possible to develop this theory in some similar way on $\widehat{\X}$ instead (apart form the obvious way by identifying $\eta$ with $\widehat{\eta}$ and lifting all the structure to $\widehat{\X}$), and see which metrics one could use.
In particular considering that the Wasserstein metric comes up (but for very different reasons) in certain recent developments in connection with analysis in metric measure spaces, such as for instance in  \cite{AGS}.   
Possibly there is a simpler description of the metric $d_{\M}$ (or some similar metric for which the above type of construction also work), which could have been easier to handle than the hands-on definition that was used in this article.\smallskip\newline
{\bf Future developments:}\smallskip\newline
The first challenge that needs to be investigated for these spaces seems to be in which situations there are plenty of curves in $\X$ so that a reasonable theory can be expected. From the results in this paper it is more or less clear that we do have plenty of curves in the following situations:
\begin{itemize}
\item weighted $\R^n$ for weight functions which locally are bounded from below by some positive constant,
\item weighted $\R^n$ for continuous weight functions,
\item Riemannian manifolds.
\end{itemize} 
However even for more difficult weight functions on $\R^n$ it is not self-evident how many curves exists.

Another obvious challenge is to see how these spaces are related to other types of Sobolev spaces such as the Newtonian ones in other settings than merely $\R^n$.

If the spaces seems to be of sufficient interest it is then also possible to look at finer properties of functions in them, Poincar\'e{}e inequalities and to develop potential theory in this setting.   

Finally it would be interesting to develop the theory in a point-free way, axiomatising the set $\X$ in a suitable manner.

\end{document}